\documentclass[reqno, 12 pt]{amsart}
\usepackage{amsmath,amssymb,latexsym}
\usepackage{hyperref}
\usepackage{amsthm}
\usepackage{color}  
\usepackage[all,cmtip]{xy}
\usepackage{adjustbox}
\usepackage{verbatim}

\hypersetup{
   colorlinks=true, 
    linktoc=all,     
    linkcolor=blue,
    citecolor=blue,  
}

\newcommand{\Z}{{\mathbb{Z}}}
\newcommand{\Q}{{\mathbb{Q}}} 
\newcommand{\N}{{\mathbb{N}}}    
    
\newcommand{\C}{{\mathbb{C}}}

\newcommand{\p}{{\mathfrak{p}}}

\newcommand{\lra}{\longrightarrow}

\newcommand{\Gal}{\mathrm{Gal}}

\newcommand{\cyc}{\mathrm{cyc}}

\newcommand{\La}{\Lambda}

\theoremstyle{plain}

\newtheorem{theorem}{Theorem}[section]
\newtheorem*{theorem*}{Theorem}

\newtheorem{proposition}[theorem]{Proposition}
\newtheorem{rem}[theorem]{Remark}

\newtheorem{lemma}[theorem]{Lemma}
\newtheorem{corollary}[theorem]{Corollary}
\newtheorem{defn}[theorem]{Definition}

\newtheorem{example}[theorem]{Example}

\begin{document}
\title{$p^r$-Selmer companion modular forms.} 
\author{Somnath Jha, Dipramit Majumdar, Sudhanshu Shekhar} \address{Somnath Jha, Department of Mathematics and Statistics, IIT Kanpur,  208016, India}\email{jhasom@iitk.ac.in} \address{Dipramit Majumdar, Department of Mathematics, IIT Madras, Chennai-600036, India} \email{dipramit@iitm.ac.in}
 \address{Sudhanshu Shekhar, Department of Mathematics and Statistics, IIT Kanpur,  208016, India}\email{sudhansh@iitk.ac.in}
\begin{abstract} {The study of $n$-Selmer group of elliptic curve over number field  in recent past has led to the discovery of some deep results in the arithmetic of elliptic curves. Given two elliptic curves $E_1$ and $E_2$ over a number field $K$, Mazur-Rubin\cite{mr} have defined them to be {\it $n$-Selmer companion} if for every quadratic twist $\chi$ of $K$, the $n$-Selmer groups of $E_1^\chi $ and $E_2^\chi$ over $K$ are  isomorphic. Given a prime $p$, they have given sufficient conditions for two elliptic curves to be $p^r$-Selmer companion in terms of mod-$p^r$ congruences between the curves. We discuss an analogue of this for Bloch-Kato $p^r$-Selmer group of modular forms.  
We compare the Bloch-Kato  Selmer groups of a modular form respectively  with the Greenberg  Selmer group when the modular form is $p$-ordinary and with the signed Selmer group of Lei-Loeffler-Zerbes when the modular form is non-ordinary at $p$. We also indicate the corresponding results over $\Q_\cyc$ and its relation with the well known congruence results of the special values of the corresponding $L$-functions due to Vatsal.}\end{abstract}
\maketitle

\section*{Introduction}
The study of $n$-Selmer group of elliptic curves over number field has been of considerable interest  in recent past. For example, for certain $n$, some striking results on the bounds on the average rank of $n$-Selmer group of elliptic curves over $\Q$ has been established  by Bhargava et. al. On the other hand, some deep results related to the rank distribution of $n$-Selmer group for certain $n$, of a family consisting of all quadratic twist  of an elliptic curve, has been studied by  Mazur-Rubin and others (cf. \cite{mr}). In \cite{mr}, instead of the rank distribution of the $n$-Selmer group over the family, 
 they formulate the inverse question: given a prime $p$ and a number field $K$, what all information about $E$ is encoded in the  $p$-Selmer group of an elliptic curve $E$ over  $K$? Motivated by this question, they define the following:

\begin{defn}\cite[Definition 1.2]{mr}
Let $K$ be a number field and $n \in \N$ be fixed. Two  elliptic curves $E_1$, $E_2$ are said to be $n$-Selmer companion, if for every quadratic character $\chi$ of $K$, there is an isomorphism of $n$-Selmer groups of $E_1^\chi$ and $E_2^\chi$ over $K$ i.e. $S_n(E_1^\chi/K) \cong S_n(E_2^\chi/K).$
\end{defn}

That naturally led them to study $p^r$-Selmer companion  elliptic curves for a  (fixed) prime $p$ with $r \in \N$. They gave sufficient conditions for two elliptic curves to be $p^r$-Selmer companion in terms of various conditions related to mod $p^r$-congruences between the curves (see main theorem \cite[Theorem 3.1]{mr}). They point out in \cite[\S 1]{mr} that it would be interesting to investigate this phenomenon more generally for $p^r$ Bloch-Kato Selmer group of motives instead of elliptic curves and that has led us to study this.

In this article, we fix a prime $p$ and study  Bloch-Kato  $p^r$-Selmer  groups associated to  modular forms  and discuss $p^r$ Selmer companion   modular forms (see Definition \ref{7489274794792834}).  
To avoid various technical difficulties which naturally arises in our case of Selmer group of modular forms, throughout the paper we make the restrictive hypothesis that the prime $p$ is odd.
Also, we state all our results for Selmer groups defined over $\Q$ (see \S \ref{sec2}). However, it can be seen from our proofs that our results can be  extended  to a general number field $K$ 
 at the cost of notation and hypothesis becoming more cumbersome but  essentially same in nature.

 A sample of our results is given in the following corollary which is a special case  of our main theorem (Theorem \ref{mainthmgr2}) for weight $k=2$ and  nebentypus $\epsilon =1$.   For  $f_1,f_2 \in S_k(\Gamma_0(N))$, let $K_{f_1,f_2}$ be the number field generated by the Fourier coefficients of $f_1, f_2$ and $\pi$ be a uniformizer of the ring of integers of the completion of  $K_{f_1,f_2}$ at a prime above $p$.

\begin{corollary}\label{cor}
We fix an odd prime $p$ and  $N \in \N$ with $(N,p)=1$. Let $i \in \{1,2\}$ and $f_i$ be a normalized cuspidal  Hecke eigenform  in $S_2(\Gamma_0(N))$.   Let $r \in \N$ and $\phi : A_{f_1}[\pi^r] \lra A_{f_2}[\pi^r]$ be a $G_\Q$ linear isomorphism. We assume that \begin{itemize}
\item $N$ is square-free and $\forall \ell \in \{\ell \text{ prime}: \ell \mid \mid N \}$, 
 $ \text{cond}_\ell(\bar{\rho}_{f_1})=\ell = \text{cond}_\ell(\bar{\rho}_{f_2})$.
\end{itemize}
Then for every quadratic character $\chi $ of $G_\Q$, we have an isomorphism 
of the 
 $\pi^r$-Bloch-Kato Selmer groups of $f_1\otimes\chi$ and $f_2\otimes\chi$ over $\Q$ i.e. 
$$S_\mathrm{BK}(A_{f_1\chi}[\pi^{r}]/\Q) \cong S_\mathrm{BK}(A_{f_2\chi}[\pi^{r}]/\Q).$$ 

\end{corollary}

\begin{rem}\label{bothsupandnonsupcases}
Note that both the cases; (i) when $f_i$ is $p$-ordinary and (ii) $f_i$ is `supersingular' i.e. non-ordinary at $p$ are covered in Corollary \ref{cor} (and also in Theorem \ref{mainthmgr2}).

Our main theorem (Theorem \ref{mainthmgr2})   includes the cases   $k >2$, nebentypus  is non-trivial. Moreover in the case when $f_i$ is $p$-ordinary, we also include the case when $f_i$ is `bad' at $p$ i.e.  the level  of $f_i =Np^{t_i}$ with $t_i \geq 0$ in our main theorem. 
\end{rem}

The basic strategy  of the proof is to compare each local factor which arise in the definition of Selmer group.  Specially, comparing the local factors of Selmer group at the prime $p$  requires most careful analysis. Note that in our case of studying Bloch-Kato Selmer group of modular forms, we do not have the advantage of considering the Kummer map on abelian variety. Also, we do not have an obvious analogue of using fppf cohomology on N\'eron model  of elliptic curve, used by Mazur-Rubin for treating the case of  local factors of Selmer group at the prime $p$. Thus we adopt a different strategy. Let $f_1$ and $f_2$ be two weight $k$ normalized cuspforms which are congruent mod $\pi^r$.  If $f_1$ and $f_2$ are $p$-ordinary, then we compare $\pi^r$-Bloch-Kato  Selmer local condition at $p$ with the $\pi^r$-Greenberg Selmer local condition at $p$. On the other hand, if $f_1$ and $f_2$ are non-ordinary at $p$, then we compare $\pi^r$-Bloch-Kato  Selmer local condition at $p$ with the $\pi^r$-signed Selmer local condition at $p$ (\cite{lh}). The theory of signed Selmer group for non-ordinary modular forms are developed using works of Lei-Loeffler-Zerbes (\cite{llz}, \cite{llzant}) and  can be viewed as generalization of $\pm$ Selmer groups of supersingular elliptic curves developed by S. Kobayashi.

Indeed, our method for local factors at primes  $\ell$ with $\ell \neq p$ is also different; we impose some condition on the conductor  of the residual mod $\ell$-Galois representation  and use some Iwasawa theoretic techniques. We have shown  in \S \ref{secexam}, this condition can be  verified in many cases; for example using {\it level lowering} results Ribet, Serre {\it et. al}.

In Definition \ref{74892747947928340} we take the liberty to extend the definition of Selmer companion for two cuspforms of different weights. Using this definition in  Corollary \ref{mainthmgr3}, we give  sufficient conditions for two  forms of two different weights (different $p$-power level and different nebentypus) to be $\pi^r$-Selmer companion.  We gives such example of extended $\pi^r$-Selmer companion forms within an ordinary Hida family in Example \ref{sjwjakkoi09k3kllm}(4).

We also discuss $\pi^r$ Selmer companion forms over $\Q_\cyc$, when the  congruent modular forms are good, ordinary at $p$. The proof over $\Q$ easily adapts  to this case. However,  there is a well known congruence result of Vatsal \cite{va} which shows if $f_1 \equiv f_2$  (mod $\pi^r$) are of weight $k \geq 2$, level $N$ and are good, ordinary at $p$, then for any Dirichlet character $\chi$ of conductor prime to  $N$, the $p$-adic $L$-functions of $f_1\otimes \chi$ and $f_2\otimes \chi$ are congruent mod $\pi^r$ in the Iwasawa algebra. Thus via the Iwasawa Main Conjecture 
our result is a  precise algebraic counterpart of Vatsal's result. However, in case of quadratic characters, our congruence result on Selmer groups works for all possible quadratic characters.

The structure of the article is as follows. After fixing the notation and basic set-up in \S \ref{sec0}, we recall the Galois representation attached to a  newform  in \S \ref{sec1} and \S \ref{sec2} contains the definitions of Greenberg,   Bloch-Kato, signed 1 and signed 2 $\pi^r$-Selmer group of a twisted nomalized cuspidal eigenform.  Congruence results on Greenberg  Selmer groups are discussed in \S \ref{secg}. In  \S \ref{secbk}, we study $\pi^r$ Bloch-Kato  Selmer companion forms and establish our main result (Theorem \ref{mainthmgr2}).  We discuss  $\pi^r$  Selmer companion forms over  $\Q_\cyc$ in  \S \ref{seccyc}, and point out its relation with the well known congruences of the corresponding $p$-adic $L$-functions. We compute several numerical examples  verifying our results (at weight 2 as well as at higher weights) in \S \ref{secexam}. In particular, in \S \ref{secexam} we have given explicit numerical example of $\pi^r$-Bloch-Kato Selmer companion modular forms which are (i) $p$-ordinary and (ii) those which are non-ordinary at $p$  as well.
\section{Preliminaries}\label{section1}
\subsection{Notation and set up:}\label{sec0} Throughout we fix an embedding $\iota_\infty$ of a fixed algebraic closure $\bar{\Q}$ of $\Q$ into $\C$ and also an embedding $\iota_\ell$ of $\bar{\Q}$ into a fixed algebraic closure $\bar{\Q}_\ell$ of the field $\Q_\ell$ of the $\ell$-adic numbers, for every prime $\ell$. Fix an odd prime $p$ and a positive  integer $N$ with $(N,p)=1$. 
 Let $i \in \{1,2\}$ and $f_i \in S_{k}(\Gamma_0(Np^{t_i}), \epsilon_i)$  where $t_i \in \N \cup \{0\}$, be a  normalized cuspidal Hecke eigenform which is a newform of conductor $Np^{t_i}$  and nebentypus $\epsilon_i$.  We assume that $\epsilon_i$ is a primitive Dirichlet character of conductor $C_i$. Then $C_i \mid Np^{t_i}$. We can write $\epsilon_i = \prod_\ell \epsilon_{i, \ell}$ where $\epsilon_{i,\ell}$ is a primitive Dirichlet character of conductor $=\ell^{n(i,\ell)}$ with $n(i,\ell) \in \N$, for every prime divisor $\ell$ of $C_i$. In particular, we can write $\epsilon_i = \epsilon_{i,p} \epsilon'_i$, where $\epsilon'_i$ is a Dirichlet character of conductor, say $C_i'$ with  $(C_i',p)=1$ and $C_i'\mid N$.  The {\it order} of a Dirichlet character $\tau:=$ the order of the subgroup of the roots of unity in $\C^\ast$ generated by the image of $\tau$. We define a condition $(\mathrm{C'_{i,\ell}})$ to be used later:
 \begin{equation}\label{conduction-condition126}
 (\mathrm{C'_{i,\ell}}): \text{ For every prime } \ell \mid C'_i, \text{ the order of } \epsilon_{i,\ell}^2 \neq p^n \text{ for any } n \in \N.
 \end{equation}
\begin{rem}\label{jsfheakfuwalkui}
For example, if the order of $\epsilon'_i$ is co-prime to $p$, then  $(\mathrm{C'_{i,\ell}})$ is satisfied. 
\end{rem}  
 
 Let $K= K_{f_1,f_2,\epsilon_1, \epsilon_2}$ be the number field generated by the Fourier coefficients of $f_1, f_2$ and the values of $\epsilon_1, \epsilon_2$. Let $\pi_K$ be a uniformizer of the ring of integers $O_K$ of completion of  $K$ at a prime  $\mathfrak p$ lying above $p$ induced by the embedding $\iota_p$. To ease the notation, we often write $O=O_K$ and $\pi=\pi_K$. Put 
 $$S = S_{f_1,f_2}:=\{ \ell \text{ prime}: \ell  \mid \mid N\}.$$
 Note that, by definition $p\not \in S$. For any  separable field $K$, $G_K$ will denote the Galois group $\text{Gal}(\bar{K}/K)$. For any $O_K$  module $M$, $M[\pi^{r}]$ will denote the set of $\pi^{r}$ torsion points of $M$. For a group $G$ acting on a module $M$, we denote by $M^G = \{ m \in M | gm =m ~ \forall g \in G\}.$ Also for a number field or a $p$-adic field  field $F$, and a discrete $\text{Gal}(\bar{F}/F)$ module $M$, $H^i(F,M)$ will denote the Galois cohomology group $H^i(\text{Gal}(\bar{F}/F), M)$.

\subsection{Galois representation of  a modular form}\label{sec1}
Let $p$ be a fixed odd prime and $A \in \N$ with $(A,p) =1$. Let $h =\sum a_n(h)q^n \in S_k(\Gamma_0(Ap^t), \psi)$ be a normalized eigenform of weight  $k \geq 2$ and  nebentypus $\psi$.    Then $h$ is    $p$-ordinary if $\iota_p (a_p(h))$ is a $p$-adic unit.
 Let $K_h$ be the number field generated by the Fourier coefficients of $h$  and the values of $\psi$.  Let $L$ be a number field containing $K_h$ and $L_{\mathfrak p}$ denote the completion of this number field at a prime  $\mathfrak p$ lying above $p$ induced by the embedding $\iota_p$. Let $O_L$ denote the ring of integers of  $L_{\mathfrak p}$ and $\pi_L$ be a uniformizer of $O_L$. We denote by $\omega_p: G_\Q \lra \Z_p^\times$ the $p$-adic cyclotomic character.
\begin{theorem}\label{rhof}[Eichler,  Shimura, Deligne, Mazur-Wiles, Wiles etc.]
Let $h=\sum a_n(h)q^n\in S_k(\Gamma_0(Ap^t),\psi)$  be a  newform of weight $k\geq 2$ where $(A,p) =1$.  Then there exists a  Galois representation,
 \begin{equation}\label{rhof1}
\rho_h: G_\Q \lra GL_2(L_{\mathfrak p}) \quad \text{ such that}
\end {equation}

\begin{enumerate}
\item \label{unr1}
at all primes $\ell \nmid Ap$, $\rho_h$ is unramified  with the characteristic polynomial of the (arithmetic) Frobenius is given by 
$$\mbox{trace}(\rho_h(Frob_\ell))=a_\ell(h),\quad\mbox{det}(\rho_h(Frob_\ell)) =\psi(\ell)\omega_p(Frob_\ell)^{k-1} 
= \psi(\ell)\ell^{k-1}.$$
 It follows (by the Chebotarev Density Theorem) that $\text{det}(\rho_h)=\psi \omega_p^{k-1}$.
\item  Let    $G_p$ denote the decomposition subgroup of $G_\Q$ at $p$. In addition, let us assume $h$ is $p$-ordinary and denote by $\alpha_p(h)$ (respectively $\beta_p(h)$) the  $p$-adic unit (respectively non $p$-adic unit) root of the polynomial $X^2 -a_p(h)X +\psi(p) p^{k-1}.$  Let   $\lambda_h$ be the unramified character with $\lambda_h(Frob_{p})=\alpha_p(h)$. Then by Mazur-Wiles, Wiles

$$ \rho_{h}|_{G_{p}} \sim \begin{pmatrix} \lambda_h^{-1} \psi \omega_{p}^{k-1}  & * \\   0 & \lambda_h \end{pmatrix},$$

\end{enumerate} 
\end{theorem}
Let $V_h \cong L_{\mathfrak p}^{\oplus 2}$ denotes the representation space of $\rho_h$. By compactness of $G_\Q$, choose an  $O_L$ lattice $T_{h}=T_h^L$ of $V_h$ which is invariant under $\rho_{h}$. 
Let 
\begin{equation}\label{rrho}
\bar{\rho}_h: {\mathrm G}_{\Q} \lra \mathrm{GL_2}(\frac{O_L}{\pi_L})
\end{equation}
 be the residual representation of $\rho_h$.

\subsection{Definition of the  Selmer groups}\label{sec2}
We choose and fix a quadratic character $\chi$ and set $M := \text{cond}(\chi).$ Recall for $i=1,2$, $f_i \in S_{k}(\Gamma_0(Np^{t_i}), \epsilon_i)$ are  two (fixed) normalized cuspidal eigneforms with $(N,p)=1$ and $t_i\in \N \cup \{0\}$.  Also recall $S = S_{f_1,f_2}:=\{\ell \text{ prime}: \ell \mid \mid N\}.$ Let $\Sigma$ be a finite set of primes of $\Q$  such that $\Sigma \supset S \cup \{p\} \cup M$. Let $\Q_\Sigma$ be the maximum algebraic extension of $\Q$ unramified outside $\Sigma$ and set $G_\Sigma(\Q) = \text{Gal}(\Q_{\Sigma}/\Q)$.  In this subsection, we use $h$ as a notation for $f_1$ or $f_2$ i.e. $h \in \{f_1, f_2\}$.  Hence we can  take $L =K_{f_1,f_2,\epsilon_1, \epsilon_2}$, $O_L =O $,  $\pi_L= \pi$.   Upon choosing $O$  lattice $T_h \subset V_h \cong K_{\mathfrak p}^{\oplus 2}$, we have an induced $G_\Q$ action on the discrete module $A_h:= V_h/T_h.$ We also have the canonical maps
 \begin{equation}\label{pif}
  0 \lra T_h \lra V_h \stackrel{p_h}{\lra} A_h \lra 0.
  \end{equation}
   For $j \in \Z$, set $V_{h\chi(-j)} =  V_h \otimes \chi  \omega_p^{-j} = V_h \otimes\Q_p( \chi  \omega_p^{-j}) $ with the diagonal action of $G_\Q$. We further define $T_{h\chi(-j)} = T_h \otimes \chi \omega_p^{-j}$ and put $A_{h\chi(-j)} = \frac{V_{h\chi(-j)}}{T_{h\chi(-j)}} $.

\par Let   $K \subset \Q_\Sigma$ be a number field. We define $\Sigma_K$ to be the set primes in $K$ lying over the primes in $\Sigma$. In particular for $K=\Q$, $\Sigma_\Q=\Sigma$. For every prime $v \in \Sigma_K $,  
let us choose a subset $H^1_\dagger (K_v,A_{h\chi(-j)}) \subset H^1 (K_v,A_{h\chi(-j)}).$ For this choice,  we  define $\dagger$-Selmer group $S_\dagger(A_{h\chi(-j)}/K)$ as 

\begin{equation}\label{sel-general}
S_\dagger(A_{h\chi(-j)}/K): =\mathrm{Ker}\Big(H^1(\Q_\Sigma/K, A_{h\chi(-j)}) \lra \underset{v \in \Sigma_K}\prod \frac{H^1(K_v, A_{h\chi(-j)})}{H_\dagger^1(K_v, A_{h\chi(-j)})}\Big) 
\end{equation}

For any $r \in \N$, there is a canonical map  
   \begin{equation}\label{ifr}
   0 \lra A_{h\chi(-j)}[\pi^r] \stackrel{i_{{h\chi(-j)},r}}{\lra} A_{h\chi(-j)}
   \end{equation}
    
Next we define $\pi^r$     $\dagger$-Selmer group $S_\dagger(A_{h\chi(-j)}[\pi^r]/K)$ as 
\begin{equation}\label{sel-mod-r-general}
S_\dagger(A_{h\chi(-j)}[\pi^r]/K): =\mathrm{Ker}\Big(H^1(\Q_\Sigma/K, A_{h\chi(-j)}[\pi^r]) \lra \underset{v  \in \Sigma_K}\prod \frac{H^1(K_v, A_{h\chi(-j)}[\pi^r])}{H_\dagger^1(K_v, A_{h\chi(-j)}[\pi^r])}\Big),
\end{equation}
where $H_\dagger^1(K_v, A_{h\chi(-j)}[\pi^r]):=  { i_{{h\chi(-j)},r}^{*^{-1}}} (H^1_\dagger (K_v,A_{h\chi(-j)}))$ for every $v  \in \Sigma_K$. Here $$i_{{h\chi(-j)},r}^{*}: H^1(K_v,  A_{h\chi(-j)}[\pi^r]) \lra H^1(K_v, A_{h\chi(-j)}) $$ is  induced from $i_{{h\chi(-j)},r}$ in \eqref{ifr}. When there is no confusion, we may denote $i_r^*:= i_{{h\chi(-j)},r}^{*} $. 

\medskip

\begin{rem}\label{rem1first}
For $K=\Q$, note that if $A_{h}[\pi]$ is irreducible  as a $G_\Q$ module then the natural Kummer map $H^1(G_\Sigma(\Q), A_{h\chi(-j)}[\pi^r])  \lra H^1(G_\Sigma(\Q), A_{h\chi(-j)})[\pi^r]$ is an isomorphism.  Also it follows from the definition of the Selmer group above that if $A_{h\chi(-j)}[\pi]$ is irreducible, then the natural map $S_\dagger(A_{h\chi(-j)}[\pi^r]/\Q) \lra S_\dagger(A_{h\chi(-j)}/\Q)[\pi^r]$ is an isomorphism as well.
\end{rem}

For a prime $v \in \Sigma_K$, let $I_v$ be the inertia subgroup of $\mathrm{Gal}(\bar{\Q}/K)$ at $v$ and  $p_{h\chi(-j)}^*: H^1(K_v, V_{h\chi(-j)}) \lra H^1(K_v, A_{h\chi(-j)})$ is  induced  from the map  $p_{h\chi(-j)}$ in \eqref{pif}. Now we will make special choice of $H^1_\dagger (K_v,A_{h\chi(-j)}) $ for every $v \in \Sigma_K$, to respectively define  Bloch-Kato Selmer group  and under suitable condition we will also define Greenberg,  signed 1 and  signed  2 Selmer groups.

 First $\dagger =\mathrm{BK} $  is  defined as follows:  For a $v \in \Sigma$, set $$
H^1_\mathrm{BK}(K_v, A_{h\chi(-j)}) : =
\begin{cases} 
 p_{h\chi(-j)}^*(H^1_{\text{unr}}(K_v, V_h))& \text{if } v \in \Sigma_K, v \nmid p \\
 p_{h\chi(-j)}^*(H^1_{f}(K_v, V_h))& \text{if } v \mid p.
\end{cases}
$$
\begin{equation}\label{unr}
 \text{where, } \quad H^1_{\text{unr}}(K_v, V_{h\chi(-j)}) =  \mathrm{Ker}\Big(H^1(K_v, V_{h\chi(-j)}) \lra H^1(I_v, V_{h\chi(-j)})\Big), ~ v \in \Sigma_K, v \nmid p 
 \end{equation}
\begin{equation}\label{finite}
 \text{and } \quad  H^1_{\text{f}}(K_v, V_{h\chi(-j)}) =  \mathrm{Ker}\Big(H^1(K_v, V_{h\chi(-j)}) \lra H^1(K_v, V_{h\chi(-j)}\otimes_{\Q_p}B_\mathrm{crys})\Big), ~ v \mid p
 \end{equation} where $B_\mathrm{crys}$ is as defined by Fontaine in \cite{fon}. This completes the definition of  Bloch-Kato Selmer groups. 
 
 For $v \mid p$, we also recall the definition of a subgroup $H^1_g(K_v,V_{h\chi(-j)})$ of $H^1(K_v,V_{h\chi(-j)})$, to be used later  
\[  H^1_g(K_v,V_{h\chi(-j)}): = \mathrm{Ker}\Big(H^1(K_v,V_{h\chi(-j)})\lra H^1(K_v,V_{h\chi(-j)}\otimes_{\Q_p} B_{dR})\Big), ~ v\mid p \]
 where $B_{dR}$  is as defined by Fontaine in \cite{fon}.

 Next we take $\dagger =\mathrm{Gr}$ and define Greenberg Selmer group.  To define this, is necessary to assume that $h$ is ordinary at $p$. Then by the $p$-ordinary property of $\rho_h$, by Theorem \ref{rhof}(2), $A_h$ has a filtration as a $G_{\Q_p}$ module
\begin{equation}\label{iaff}
 0 \lra A_h' \lra A_h \lra A_h^{''} \lra 0,
 \end{equation}
 where  both ${A_h'}^\vee$ and  ${A_h^{''}}^\vee$ are free  $O$ module of rank 1 and the action of $G_\Q$ on $A_h^{''}$ is unramified at $p$. We can also get a similar filtration on $A_{h\chi(-j)} $. We now  define 
 
 $$
H^1_\mathrm{Gr}(K_v, A_{h\chi(-j)}) : =
\begin{cases} 
\text{Ker}\big(H^1(K_v, A_{h\chi(-j)}) \lra H^1(I_v, A_{h\chi(-j)})\big) & \text{if } v \in \Sigma_K,  v \nmid p \\
\text{Ker}\big(H^1(K_v, A_{h\chi(-j)}) \lra H^1(I_v, A^{''}_{h\chi(-j)})\big) & \text{if } v \mid p.
\end{cases}
$$

Now  we take $\dagger =i$ and define signed $i$ Selmer group for $i \in \{1,2\}$ (cf. \cite{llzant}, \cite{lh}). To define this, it is necessary to assume that $h$ is `good' at $p$ and also non-ordinary at $p$ i.e. $t_1=t_2=0$, $ (N, p)$ =1 and $v_p(a_p(h)) \neq 0$.  We also assume that the quadratic character $\chi $ is trivial i.e. we only define signed Selmer group  for $f\otimes\omega_p^{-j}$. Further, we assume $K \subset \Q(\mu_{p^\infty}) :=\underset{n}{\cup}\Q(\mu_{p^n})$.  For a prime $v \in \Sigma, v \nmid p$, we simply define $$
H^1_\mathrm{i}(K_v, A_{h(-j)}) : =  H^1_\mathrm{BK}(K_v, A_{h(-j)}), \quad i=1,2.$$
For $v \mid p$, we will now define  $H^1_i(K_v, A_{h(-j)}) $ for $i=1,2$ respectively.

 As $v_p(a_p(h)) \neq 0$, a pair of Coleman maps $$
\mathrm{Col}_{h,i}:H^1_\mathrm{Iw}(\Q_p,T_h) \cong \underset{n}{\varprojlim}H^1(\Q_p(\mu_{p^n}), T_h) \lra O[\Delta][[\Gamma]]\cong O[\Delta][[T]]
$$ are defined (see \cite[\S2]{lh} \cite{llzant} for details). Here $\Delta=\mathrm{Gal}(\Q(\mu_p)/\Q) $ and $O=O_h$.
 Let Pr$_{K_v}$ be the natural projection map 
 \begin{equation}\label{djhqhd}
 H^1_\mathrm{Iw}(\Q_p, T_{h(-j)}) \stackrel{{Pr}_{K_v}}{\lra} H^1(K_v, T_{h(-j)}).
 \end{equation} For $i=1,2$, we first define a subset $H^1_i(K_v, T_{h(-j)}) \subset H^1(K_v, T_{h(-j)})$ to be  $$H^1_i(K_v, T_{h(-j)}) :=\text{Pr}_{K_v}(\text{Ker}(\text{Col}_{h,i})\otimes \omega_p^{-j}),\quad ~ i=1,2$$  

Now consider the natural map $\iota:H^1(K_v, T_{h(-j)}) \lra H^1(K_v, V_{h(-j)})$ induced by the inclusion of $T_h$ in $V_h$ and denote by $V_i$ the subspace of $H^1(K_v, V_{h(-j)})$  generated by the image of $H_i^1(K_v, T_{h(-j)}) $ under $\iota$ i.e. $V_i=<\iota(H_i^1(K_v, T_{h(-j)}))>$. Finally, define $H^1_i(K_v, A_{h(-j)}) 
: =\mathrm{Proj}(V_i)$ where Proj is the natural map $\mathrm{Proj}:H^1(K_v, V_{h(-j)}) \lra H^1(K_v, A_{h(-j)})$ induced by the projection of $V_h$ to $A_h$. This completes the definition of signed 1 and 2 Selmer groups of $h\otimes \omega_p^{-j}$.

Finally, we define $\pi^r$  Selmer companion  modular forms.

\begin{defn}\label{7489274794792834}
Let $i \in \{1,2\}$ and  $f_i\in S_k(\Gamma_0(N_i), \epsilon_i)$ be a normalized cuspidal eigenform where $N_i \in \N$. Let $r \in \N$. 
We say $f_1$ and $f_2$ are $\pi^r$  (Bloch-Kato) Selmer companion, if for each critical twist $j$ with $0\leq j \leq k-2$ and for every quadratic character $\chi$ of $G_\Q$, we have an isomorphism of $\pi^r$ Bloch-Kato Selmer groups of $f_1\otimes \chi \omega_p^{-j}$ and $f_2\otimes \chi \omega_p^{-j}$ over $\Q$ i.e.
$$S_\mathrm{BK}(A_{f_1\chi(-j)}[\pi^{r}]/\Q) \cong S_\mathrm{BK}( A_{f_2\chi(-j)}[\pi^{r}]/\Q). $$
\end{defn}

  Note that for weight $k \geq 2$, corresponding to $k-1$ critical values, there are  $k-1$ many Selmer groups associated to a cuspidal eigenform; and for $f_1$ and $f_2$ to be  $\pi^r$-Selmer companion, we need each $j $ with $0 \leq j \leq k-2$ and every $\chi$, $S_\mathrm{BK}(A_{f_1\chi(-j)}[\pi^{r}]/\Q) \cong S_\mathrm{BK}(A_{f_2\chi(-j)}[\pi^{r}]/\Q)$. 
  
\begin{rem}\label{finintefirst88237}
\rm{
Mazur  and Rubin \cite[\S 1]{mr} have stated that given an elliptic curve $E$ over a number field $K$, they expect  there are only finitely many elliptic curves $E'$ over $K$ such that $E$ and $E'$ are Selmer companion.  Moreover they  have shown in \cite[Proposition 7.1]{mr}  that given an elliptic curve $E$ over $K$ there are only finitely many elliptic curves $E'$ over $K$ such that the pair $(E,E')$ satisfy all the conditions of their main theorem (Theorem 3.1, \cite{mr}).

Our situation is different. Let  $f\in S_2(\Gamma_0(N))$ be a $p$-ordinary  normalized Hecke eigen newform with $a_p(f) \neq \pm 1$ (mod $\pi$), $N$ is squarefree and coprime to $p$. We further assume that cond$_\ell(\bar{\rho}_f)=\ell$ for every prime $\ell||N$. (In Example \ref{sjwjakkoi09k3kllm}(1), we have given explicit example of such an $f$.)

Then by Hida theory (see \cite{Wi}) there exists  infinitely many $f_r\in S_2(\Gamma_0(Np^r), \psi_r)$ such that $f_r \equiv f$ (mod $\pi$), where $\pi  =\pi_{f,f_r, \psi_r}$ and $\psi_r$ is certain  dirichlet character of  conductor $p^{r}$ satisfying $\psi_r$ is trivial (mod $\pi)$. 
 Then for each $r$, 
 the pair $(f,f_r)$ satisfies conditions $(1),(2)$ and $(3)$ of Theorem \ref{mainthmgr2}. In particular, $f$ and $f_r$ are $\pi$ Selmer companion for infinitely many $f_r$. 
}    
\end{rem}  
\begin{rem}\rm{
In the converse direction, Mazur and Rubin have asked if two elliptic curves $E$ and $E'$ over $K$ are $p^r$-Selmer companion, then can we say that $E[p^r]\cong E'[p^r]$ as $\Gal(\bar{K}/K)$-modules (see \cite[Conjecture 7.14]{mr},  in arxiv version.)? It would be interesting to investigate if $\pi^r$-Selmer companion modular forms of weight $k\geq 2$ are congruent mod $\pi^r$. }
\end{rem}

\section{ `Greenberg Selmer companion forms'}\label{secg}
In this section, we compare twisted $\pi^r$-Greenberg Selmer group of two $p$-ordinary $\pi^r$ congruent cuspforms. We use these results in the next section to study   Bloch-Kato Selmer companion forms. Our main result  in this section is the following: 
\begin{theorem}\label{mainthmgr}
Let $p$ be an odd prime and for $i=1,2$, let $f_i$ be a $p$-ordinary normalized  eigenform  in $S_k(\Gamma_0(Np^{t_i}), \epsilon_i)$, where $(N,p)=1$, $k \geq 2$,  $t_i\in \N\cup \{0\}$. Recall from \S \ref{sec0}, $C_i'$ is the tame conductor of $\epsilon_i$. Let  $r \in \N$ and $\phi : A_{f_1}[\pi^r] \lra A_{f_2}[\pi^r]$  be a $G_\Q$ linear isomorphism. We assume  the following:
\begin{enumerate}

\item $N$ is square-free and $\forall \ell \in S$,  $ \text{cond}_\ell(\bar{\rho}_{f_1})=\ell = \text{cond}_\ell(\bar{\rho}_{f_2})$.
\item The condition $(\mathrm{C'_{i,\ell}})$, defined in equation \eqref{conduction-condition126}, is satisfied for $i=1,2$.
\item Assume  $ \omega_p^{k-1}\epsilon_{i,p}  \neq 1$  (mod $\pi$) for $i=1,2$. 
\end{enumerate}
Then for each fixed $j$ with $0 \leq j \leq k-2$, and for every quadratic character $\chi $ of $G_\Q$ there is an isomorphism 
$$S_\mathrm{Gr}( A_{f_1\chi(-j)}[\pi^{r}]/\Q) \cong S_\mathrm{Gr}(A_{f_2\chi(-j)}[\pi^{r}]/\Q). $$
\end{theorem}
Here  $ \text{cond}_\ell(\bar{\rho}_{f_1}) $ and $ \text{cond}_\ell({\rho}_{f_2}) $ are defined following Serre and can be found in \cite[\S 1, Page 135]{li}. The proof of the theorem is divided into several lemmas and propositions. 
\begin{proposition}\label{conductor-equivalence-at-l}
Let $i \in \{1,2\}$ and $f_i \in S_k(\Gamma_0(Np^{t_i}), \epsilon_i)$ with  $(N,p)=1$. Assume $N$ is square-free and $\forall \ell \in S$,  $ \text{cond}_\ell(\bar{\rho}_{f_i})=\ell$. Also assume the hypothesis $(\mathrm{C'_{i,\ell}})$ defined in equation \eqref{conduction-condition126}.  Then for every quadratic character $\chi $ of $G_\Q$, for any $j \in \Z$ and for every prime $\ell \in \Sigma \setminus \{p\}$, $\mathrm{cond}_\ell (\rho_{f_i  \chi (-j)}) = \mathrm{cond}_\ell (\bar{\rho}_{f_i  \chi (-j)})$. 
\end{proposition}
\begin{proof}
Let $i \in \{1,2\}.$ Recall, $M =$ conductor of $\chi$ and $C_i= $ conductor of $\epsilon_i$. Then $M=D$ or $4D$ with $D$ square-free. Note as $\omega_p$ is unramified at $\ell$, $\mathrm{cond}_\ell (\rho_{f_i  \chi (-j)}) =\mathrm{cond}_\ell (\rho_{f_i  \chi})$  and $\mathrm{cond}_\ell (\bar{\rho}_{f_i \chi (-j)}) = \mathrm{cond}_\ell (\bar{\rho}_{f_i  \chi}).$ Thus it suffices to show that  for every quadratic character $\chi $ of $G_\Q$,  $\mathrm{cond}_\ell (\rho_{f_i \chi  }) = \mathrm{cond}_\ell (\bar{\rho}_{f_i  \chi}).$ We prove this by considering various cases.

\medskip

\underline{Case 1} $\ell || N$ and  $\ell \nmid M$: \\
Since $\ell\nmid M$, $\chi$ is unramified and we have $\mathrm{cond}_\ell (\rho_{f_i \chi }) = \mathrm{cond}_\ell (\rho_{f_i})$ and  $\mathrm{cond}_\ell (\bar{\rho}_{{f_i}  \chi }) = \mathrm{cond}_\ell (\bar{\rho}_{f_i})$. Now from the first assumption of the proposition, we have $\mathrm{cond}_\ell (\rho_{f_i})= \mathrm{cond}_\ell (\bar{\rho}_{f_i})$. Therefore $\mathrm{cond}_\ell (\rho_{{f_i}  \chi }) = \mathrm{cond}_\ell (\bar{\rho}_{{f_i}  \chi}).$

\medskip

\underline{Case 2} $\ell \nmid N$ and  $\ell \mid M$: \\
Since $\ell\nmid N$, $\rho_{f_i}$ is unramified. Therefore  $\mathrm{cond}_\ell (\rho_{f_i  \chi})= \mathrm{cond}_\ell (\chi)$ and $\mathrm{cond}_\ell (\bar{\rho}_{f_i \chi})=  \mathrm{cond}_\ell (\bar{\chi})$ where $\bar{\chi}$ denote the residual character mod $\pi$ associated to $\chi$. Since $\chi$ is a quadratic character and $p\neq 2$, $\mathrm{cond}_\ell (\chi)= \mathrm{cond}_\ell (\bar{\chi})$. This implies that $\mathrm{cond}_\ell (\rho_{f_i \chi  }) = \mathrm{cond}_\ell (\bar{\rho}_{f_i \chi }).$

\medskip

\underline{Case 3} $\ell || N, \ell \nmid C'_i$ and  $ \ell | M$:\\
By our assumption,   $\mathrm{cond}_\ell (\rho_{f_i}) = \mathrm{cond}_\ell(\bar{\rho}_{{f_i}})=\ell$.  Then from \cite[Page 135]{li}, we get
$$ \mathrm{cond}_\ell(\rho_{{f_i}}) = \ell^{\mathrm{codim} ~\rho_{f_i}^{I_\ell} + sw(\rho_{{f_i}}^{ss})} \quad \text{ and }  \quad
 \mathrm{cond}_\ell(\bar{\rho}_{{f_i}}) = \ell^{\mathrm{codim}~ \bar{\rho}_{f_i}^{I_\ell} + sw(\bar{\rho}_{{f_i}})}  $$
In particular, $\mathrm{codim}~ \bar{\rho}_{{f_i}}^{I_\ell}\leq 1$ and therefore $\mathrm{dim}~ \bar{\rho}_{{f_i}}^{I_\ell} \geq 1$. Since $\bar{\rho}_{{f_i}}$ is ramified at $\ell$,  $\mathrm{dim}~ \bar{\rho}_{{f_i}}^{I_\ell} = 1$. This implies that $\bar{\rho}_{{f_i}}$ has an unramified submodule, say $\bar{V}_1$. Put $\bar{V}_2:= \bar{V}_{{f_i}}/\bar{V}_1$ where $\bar{V}_{{f_i}}$ denote the vector space corresponding to $\bar{\rho}_{f_i}$. Let $\bar{\chi}_1$ and $\bar{\chi}_2$ be the characters associated to $\bar{V}_1$ and $\bar{V}_2$ respectively.  
As $\ell \nmid C'_i$, we get  $\bar{\chi}_1\bar{\chi}_2=\bar{\omega}_p^{k-1}\bar{\epsilon}_i$. This implies that $\bar{\chi}_2$ is also unramified.  Thus we have $ \bar{\rho}_{{f_i}} \sim  \begin{pmatrix} \bar{\chi}_1  & * \\  0 & \bar{\chi}_2 \end{pmatrix} $ and consequently\[ \bar{\rho}_{{f_i}\chi} = \bar{\rho}_{{f_i}}\otimes{\bar{\chi}} \sim  \begin{pmatrix} \bar{\chi}_1\bar{\chi}  & * \\ 0 & \bar{\chi}_2\bar{\chi}  \end{pmatrix}.\]Now  $\chi$ being quadratic and  $\ell \mid M$,  both $\chi$ and $\bar{\chi}$ are ramified  ramified at $\ell$. In particular, 
  $(\bar{V}_{{f_i}} \otimes \overline{\chi })^{I_\ell} =0$ i.e.   $\mathrm{codim} ~(\bar{V}_{{f_i}} \otimes \overline{\chi })^{I_\ell}= 2$.
On the other hand, 
$$\rho_{{f_i}  \chi}\mid_{G_\ell}~\sim ~\begin{pmatrix}
\eta_i\omega_{p}\chi &*\\
0 &\eta_i\chi \end{pmatrix},$$
where $\eta_i$ is an unramified character (see  \cite[Theorem 3.26(3)]{hi}). Again as  $\chi$ is  ramified at ${\ell}$, we deduce $(V_{{f_i}  \chi})^{I_\ell} =0$; in other words,  $\mathrm{codim}~ (V_{{f_i} \chi})^{I_\ell}= 2$ as well. Further by \cite[Prop. 1.1]{li},  $sw(\bar{V}_{{f_i}} \otimes \overline{\chi }) = sw((V_{{f_i}\chi})^{ss})$. Hence $\mathrm{cond}_\ell (\rho_{{f_i}  \chi}  ) = \mathrm{cond}_\ell (\bar{\rho}_{{f_i} \chi })$ holds true in this case.

\medskip

\underline{Case 4} $\ell || N, \ell \mid C'_i$ and  $ \ell | M$: \\
In this case, ${{\rho}_{f_i}}_{\mid_{I_\ell}} \sim  \begin{pmatrix} \epsilon_{i,\ell}   & 0 \\ 0 & 1 \end{pmatrix}$ and  
$\bar{\rho}_{{f_i}_{\mid_{I_\ell}}} \sim \begin{pmatrix}  \bar{\epsilon}_{i,\ell}   & 0 \\ 0 & 1 \end{pmatrix}$ (see \cite[Theorem 3.26(3)]{hi}). Similarly, $${{\rho}_{f_i  \chi}}_{\mid_{I_\ell}} \sim \begin{pmatrix} \epsilon_{i,\ell}\chi   & 0 \\ 0 & \chi \end{pmatrix} \quad \text{ and } \quad \bar{\rho}_{{f_i \chi}_{\mid_{I_\ell}}} \sim \begin{pmatrix}  \bar{\epsilon}_{i,\ell}\bar{\chi}   & 0 \\ 0 & \bar{\chi}  \end{pmatrix}.$$ Note that as $\ell \mid M$,  both $\chi$ and  $\bar{\chi}$ are ramified at $\ell$ as in the previous case.  

\noindent First we consider the subcase when $\bar{\epsilon}_i\bar{\chi}$ is ramified at $\ell$. This implies that $\epsilon_i \chi$ is also ramified. Thus   $\mathrm{codim}~ (\bar{V}_{{f_i}} \otimes \overline{\chi })^{I_\ell}= 2 = \mathrm{codim}~ (V_{f_i \chi})^{I_\ell}$ holds.

\noindent Next we assume $\bar{\epsilon}_i\bar{\chi}$ is unramified i.e. ${\bar{\epsilon}_{i,\ell}\bar{\chi}_\ell}$ is trivial.
 Then we have ${\bar{\epsilon}_{i,\ell}}^2$ is trivial. Now by the second assumption of the proposition, the order of $\epsilon_{i,\ell}^2 $ is not a positive power of  $p$. Hence ${\bar{\epsilon}_{i,\ell}}^2$ is trivial gives $\epsilon^2_{i,\ell
 }$ is also trivial.
 Thus  $\epsilon_{i,\ell}$ and ${\epsilon_{i,\ell} \chi_\ell}$ are both quadratic characters. Hence  ${\bar{\epsilon}_{i,\ell}\bar{ \chi}_\ell}$ is trivial implies ${\epsilon_{i,\ell} \chi_\ell}$ is trivial.
 Then using ${\rho}_{{f_i \chi}_{\mid_{I_\ell}}} \sim \epsilon_{i,\ell}\chi_\ell \oplus \chi_\ell$ and   $\chi$ is ramified at $\ell$, we deduce $\mathrm{codim} ~ (\bar{V}_{f_i} \otimes \overline{\chi })^{I_\ell}= 1 = \mathrm{codim} ~(V_{f_i  \chi } )^{I_\ell}$. Again from  \cite[Prop. 1.1]{li},  $sw(\bar{V}_{f_i} \otimes \overline{\chi }) = sw((V_{f_i\chi} )^{ss})$, hence we obtain $\mathrm{cond}_\ell (\rho_{f_i  \chi  }) = \mathrm{cond}_\ell (\bar{\rho}_{f_i  \chi })$, as required.
\end{proof}

\begin{corollary}\label{cor3467}
We keep the hypotheses of Proposition \ref{conductor-equivalence-at-l}. Then for  every quadratic character $\chi $ of $G_\Q$ and for the  $(-j)^{th}$ Tate twist of $f_i\otimes\chi$, we have $A^{I_\ell}_{f_i \chi (-j)}$ is divisible. 
\end{corollary}
\proof This is a direct consequence of Proposition \ref{conductor-equivalence-at-l} and the  proof of  \cite[Lemma 4.1.2]{epw}.  \qed
\begin{lemma}\label{sel5971}
Let us keep the hypotheses of Proposition \ref{conductor-equivalence-at-l}. Then for every $q \in \Sigma \setminus  \{p\}$, every $\chi$ and every $j$, we have 
$$H^1_\mathrm{Gr}(\Q_q,  A_{f_i\chi(-j)}[\pi^r])= \text{Ker}\Big(H^1(\Q_q,  A_{f_i\chi(-j)}[\pi^r]) \lra H^1(I_q,  A_{f_i\chi(-j)}[\pi^r])\Big).$$
\end{lemma}
\proof Consider the commutative diagram 

\xymatrix{
 H^{1}(\Q_{q},  A_{{f_i}\chi(-j)}[\pi^r]) \ar[r]^{\phi_r}\ar[d]^{i_r^*} & H^{1}(I_q,  A_{{f_i}\chi(-j)}[\pi^r]) \ar[d]^{s_r}\\
 H^{1}(\Q_{q},  A_{{f_i}\chi(-j)}) \ar[r]^{\phi} & H^{1}(I_q,  A_{{f_i}\chi(-j)})
}

\noindent Now $b \in H^1_\mathrm{Gr}(\Q_q,  A_{{f_i}\chi(-j)}[\pi^r])$  if and only if $ b \in \text{Ker}(\phi\circ i_r^*) =\text{Ker}( s_r\circ \phi_r)$. As $s_r$ is induced by the Kummer map, $\text{Ker}( s_r) \cong  \frac{ (A_{{f_i}\chi(-j)})^{I_q}}{\pi^r({A_{{f_i}\chi(-j)}})^{I_q}}$.  It follows from Corollary \ref{cor3467} that $\text{ker}(s_r) =0$ which proves the lemma. \qed

\medskip

Using Lemma \ref{sel5971}, the following expression  of  $S_\mathrm{Gr}(A_{{f_i}\chi(-j)}[\pi^r]/\Q)$ is immediate. 

\medskip

\begin{lemma}\label{sel7892}
We keep the hypotheses of Proposition \ref{conductor-equivalence-at-l}. Then we have an exact sequence \begin{footnotesize}{
$$ 0 \rightarrow S_\mathrm{Gr}(A_{f_i\chi(-j)}[\pi^r]/\Q) \lra H^1(G_\Sigma(\Q), A_{f_i\chi(-j)}[\pi^r]) \lra \underset {q \in \Sigma, q \neq p}{\prod} H^1(I_q,  A_{f_i\chi(-j)}[\pi^r])\times \frac{H^1(\Q_p,  A_{{f_i}\chi(-j)}[\pi^r])}{H^1_{Gr}(\Q_p,  A_{{f_i}\chi(-j)}[\pi^r])}.
$$}\end{footnotesize}
\end{lemma}

\medskip

Next we study  the $p$-part of the local terms defining $S_\mathrm{Gr}(A_{{f_i}\chi(-j)}[\pi^r]/\Q)$.

\medskip

\begin{proposition}\label{ipandip'}
Set $I_p' =\mathrm{Gal}(\bar{\Q}_p/ \Q^{\text{unr}}_p(\mu_{p^\infty}))$ and  $i \in \{f_1,f_2\}$. Then 
\begin{footnotesize}{
$$
H^1_\mathrm{Gr}(\Q_p,  A_{{f_i}\chi(-j)}[\pi^r]) =
\begin{cases} 
\text{Ker}\Big(H^1(\Q_p,   A_{{f_i}\chi(-j)}[\pi^r]) \lra H^1(I_p,A''_{{f_i}\chi(-j)}[\pi^r])\Big) & \text{ if  } (\omega_p^{-j}  \chi)_{\mid_{I_p}} \neq 1 (\text{mod } \pi)\text{ or } j =0. \\
 \text{Ker}\Big(H^1(\Q_p,   A_{{f_i}\chi(-j)}[\pi^r]) \lra H^1(I'_p, A''_{{f_i}\chi(-j)}[\pi^r])\Big) &  \text{ otherwise}.
\end{cases}
$$
}\end{footnotesize}
\end{proposition}
\proof \underline{Case 1}: Either $(\omega^{-j}\chi)_{\mid_{I_p}}   \neq  1 ~(\text{mod } \pi)\text{ or } j =0.$ \\
Consider the commutative diagram 
\xymatrix{
 H^{1}(\Q_{p},  A_{{f_i}\chi(-j)}[\pi^r]) \ar[r]^{\theta_r}\ar[d]^{i_r^*} & H^{1}(I_p,  A_{{f_i}\chi(-j)}^{''}[\pi^r]) \ar[d]^{i_r^{''*} }\\
 H^{1}(\Q_{p},  A_{{f_i}\chi(-j)}) \ar[r]^{\theta} & H^{1}(I_p,  A_{{f_i}\chi(-j)}^{''})
}

It suffices to show in   Case 1 that $H^1_\mathrm{Gr}(\Q_p,  A_{{f_i}\chi(-j)}[\pi^r])  = \mathrm{Ker}(\theta_r)$. From the above  diagram, we observe that $b \in H^1_\mathrm{Gr}(\Q_q,  A_{{f_i}\chi(-j)}[\pi^r])$  if and only if $ b \in \mathrm{Ker}(\theta \circ i_r^*) =\mathrm{Ker}( i_r^{''*}\circ\theta_r)$. Thus it further reduces to show  $\mathrm{Ker}(i_r^{''*}) =0$ in this case. Note that  $\mathrm{Ker}( i_r^{''*}) \cong   (A_{{f_i}\chi(-j)}^{''})^{I_p}/{\pi^r(A_{{f_i}\chi(-j)}^{''})^{I_p}}$. We divide the proof in three subcases. 

First, let  $j =0$ and $\chi$ is unramified at $p$: Then $ A''^{I_p}_{{f_i}\chi(-j)}= A''_{{f_i}\chi}$ is divisible, whence $\mathrm{Ker}(i_r^{''*}) =0$.

Second, let $j=0$ and $\chi$ is ramified: Here being a quadratic character $\chi ~(\text{mod } \pi) $ is also ramified (here we use  $p$ is odd).  Thus $I_p$ acts non-trivially on $ A''_{{f_i}\chi}[\pi]$ and hence $(A''_{{f_i}\chi}[\pi])^{I_p} =0 = (A''_{{f_i}\chi})^{I_p}$ and consequently  $\mathrm{Ker}(i_r^{''*}) =0$. 

Finally, let $j > 0$ and $(\omega^{-j}\chi)_{\mid_{I_p}}   \neq  1 ~(\text{mod } \pi)$: In this scenario, $(\omega_p^{-j}  \chi)_{\mid_{I_p}} (\text{mod } \pi)$ is  nontrivial. Hence $I_p$ again acts non-trivially on $A''_{{f_i}\chi(-j)}[\pi]$ and hence $(A''_{{f_i}\chi(-j)}[\pi])^{I_p} =0 = (A''_{{f_i}\chi(-j)})^{I_p}$ and again we conclude $\mathrm{Ker}(i_r^{''*}) =0$. 

\medskip

\underline{Case 2}:  $j > 0$ and simultaneously $(\omega_p^{-j}  \chi) _{\mid_{I_p}} =1~ (\text{mod } \pi)  $.
We now consider the following commutative diagram
\xymatrix{
 H^{1}(\Q_{p},  A_{{f_i}\chi(-j)}[\pi^r]) \ar[r]^{\psi_r}\ar[d]^{i_r^*} & H^{1}(I_p',  A_{{f_i}\chi(-j)}^{''}[\pi^r]) \ar[d]^{i_r^{''*} }\\
 H^{1}(\Q_{p},  A_{{f_i}\chi(-j)}) \ar[r]^{\psi} & H^{1}(I_p',  A_{{f_i}\chi(-j)}^{''})
}

\noindent From this diagram, it suffices to show that $H^1_\mathrm{Gr}(\Q_p,  A_{{f_i}\chi(-j)}[\pi^r])  = \mathrm{Ker}(\psi_r)$ to complete the proof of the Lemma. As $j >0$, we have $ (A''_{{f_i}\chi(-j)})^{I_p} $ is finite. Moreover, as $I_p/{I_p'} = \text{Gal}(\Q^{\text{unr}}_p(\mu_{p^\infty})/\Q^{\text{unr}}_p)$ is pro-cyclic,  $H^1(I_p/{I'_p}, A''^{I_p'}_{{f_i}\chi(-j)})$ is finite as well. Now the second assumption  $(\omega_p^{-j}  \chi)_{\mid_{I_p}} =1~ (\text{mod } \pi)  $ implies that ${I'_p}$ acts trivially on ${A''_{{f_i}\chi(-j)}}$. Hence $H^1(I_p/{I'_p}, {A''_{{f_i}\chi(-j)}}^{I'_p})= H^1(I_p/{I'_p}, {A''_{{f_i}\chi(-j)}})$ is divisible also. Hence $H^1(I_p/{I'_p}, {A''_{{f_i}\chi(-j)}}^{I'_p}) =0 $. Therefore the natural restriction map $H^{1}(I_p,  A_{{f_i}\chi(-j)}^{''}) \lra H^{1}(I_p',A_{{f_i}\chi(-j)}^{''})$ is injective. Thus we have shown that $$H^1_\mathrm{Gr}(\Q_p,  A_{{f_i}\chi(-j)}) := \mathrm{Ker}\Big(H^1(\Q_p,  A_{{f_i}\chi(-j)}) \lra H^{1}(I_p,  A_{{f_i}\chi(-j)}^{''})\Big) = \mathrm{Ker}(\psi).$$  On the other hand, divisibility of $A''^{I'_p}_{{f_i}\chi(-j)} = A''_{{f_i}\chi(-j)}$ implies that $i_r^{''*}$ is injective. Now by an argument similar to Case 1, we get that $H^1_\mathrm{Gr}(\Q_p,  A_{{f_i}\chi(-j)}[\pi^r])  = \mathrm{Ker}(\psi_r)$. \qed

\medskip

\par From the discussions in Case 1  of Proposition \ref{ipandip'}, we deduce the following corollary.

\medskip

\begin{corollary}\label{78914121881}
Assume that either (i) $j =0$ or 
(ii) $j >0$ and $\omega_p^{-j}  \chi_{\mid_{I_p}} \neq 1 ~ (\text{mod } \pi)$. Then for $i =1,2$, $A''^{I_p}_{{f_i}\chi(-j)}$ is $\pi$-divisible. \qed
\end{corollary}

\medskip

\begin{rem}\label{div}
\rm{
We have $ 0 \leq j \leq k-2$.  If we choose $\frac{p-1}{2} > k-2$ i.e. $p > 2k -3$, then $j< \frac{p-1}{2}$ and $\omega_p^j  ~(\text{mod } \pi)$ is not a quadratic character. In particular, we have $(\omega_p^{-j} \chi)_{\mid_{I_p}} \neq 1 ~(\text{mod } \pi)$ and conditions (i) and (ii) of Corollary \ref{78914121881} are satisfied.}

\end{rem}

\medskip

\begin{corollary}\label{6897245} 
We keep the hypotheses of Proposition \ref{conductor-equivalence-at-l}. Then it follows from Lemma \ref{sel5971},  Lemma \ref{sel7892} and Proposition \ref{ipandip'}  that:\\

If either $(\omega_p^{-j}  \chi)_{\mid_{I_p}} \neq 1 ~ (\text{mod } \pi)$ or when  $j =0$, then
\begin{footnotesize}{
$$
S_\mathrm{Gr}( A_{{f_i}\chi(-j)}[\pi^r]/\Q) =  \text{Ker}\Big(H^1(G_\Sigma(\Q),   A_{{f_i}\chi(-j)}[\pi^r]) \lra \underset {q \in \Sigma, q \neq p}{\prod} H^1(I_q,  A_{{f_i}\chi(-j)}[\pi^r])\times {H^1(I_p,  A''_{{f_i}\chi(-j)}[\pi^r])}\Big). 
$$}\end{footnotesize}
In other case i.e.  when $(\omega_p^{-j} \chi)_{\mid_{I_p}} =1 ~(\text{mod } \pi)$ as well as $j > 0$,
 \begin{footnotesize}{$$
S_\mathrm{Gr}( A_{{f_i}\chi(-j)}[\pi^r]/\Q) = \text{Ker}\Big(H^1(G_\Sigma(\Q),   A_{{f_i}\chi(-j)}[\pi^r]) \lra \underset {q \in \Sigma, q \neq p}{\prod} H^1(I_q,  A_{{f_i}\chi(-j)}[\pi^r])\times {H^1(I'_p,  A''_{{f_i}\chi(-j)}[\pi^r])}\Big)
$$
}\end{footnotesize}
\end{corollary}

\medskip

\begin{lemma}\label{phi''copy}
 Recall $\phi : A_{f_1}[\pi^r] \lra A_{f_2}[\pi^r]$  is the given $G_\Q$ linear isomorphism in Theorem  \ref{mainthmgr}. Let $k$ be such that  $ \omega_p^{k-1}\epsilon_{i,p} \neq 1 ~ (\text{mod }\pi)$. Then $\phi|_{G_{p}}$ induces  an isomorphism $A''_{f_1}[\pi^{r}] \cong A''_{f_2}[\pi^{r}]$. Consequently, for all quadratic character $\chi$ and all $0 \leq j \leq k-2$, $\phi|_{G_p}$ induces isomorphisms $A''_{{f_1} \chi{(-j)}}[\pi^{r}] \cong A''_{{f_2} \chi {(-j)}}[\pi^{r}]$.
\end{lemma}
\proof Let $i\in \{1,2\}.$ First we claim that 
 \begin{equation}\label{jdgqjdhgkjahdaud}
H_0(I_p,A'_{f_i}[\pi^r]) =0.
\end{equation}   This is proved by induction on $r$. The case $r=1$ is proved  as follows:  the action of $I_p$  on the one dimensional $O_f/\pi$ vector space $A'_{f_i}[\pi]$ is via $\omega_p^{k-1}\epsilon_{i,p}$ (mod $\pi$) and thus by given condition  on $k$ in the hypothesis, this action is non-trivial.  Hence $H_0(I_p, A'_{f_i}[\pi])=0 $. Then we apply induction using the short exact sequence $0 \rightarrow A'_{f_i}[\pi] \lra A'_{{f_i}} [\pi^{r}] \lra A'_{{f_i}}[\pi^{r-1}] \rightarrow 0$ to establish the claim for a general $r$. 

Next we consider the exact sequence $0 \rightarrow A^{'}_{f_i}[\pi^r] \lra A_{f_i}[\pi^r] \lra A^{''}_{f_i}[\pi^r] \rightarrow 0.$ Then using equation \eqref{jdgqjdhgkjahdaud}, we deduce
\begin{equation}\label{shcsqvcjhwqgfjwg}
H_0(I_p,  A^{''}_{f_i}[\pi^r]) \cong H_0(I_p, A_{f_i}[\pi^r]).
 \end{equation} 

Notice that as $I_p$ acts on $A_{f_i}^{''}$ trivially,  we have the identification 
\begin{equation}\label{wbdxwhdwj}
A^{''}_{f_i}[\pi^r]= H_0(I_p, A^{''}_{f_i}[\pi^r]).
\end{equation}
Using  \eqref{shcsqvcjhwqgfjwg} and \eqref{wbdxwhdwj}, we finally get the required $G_p$ linear isomorphism
$$A^{''}_{f_1}[\pi^r]  \cong H_0(I_p, A_{f_1}[\pi^r])  \cong H_0(I_p,A_{f_2}[\pi^r]) \cong A^{''}_{f_2}[\pi^r]. \qed$$

\medskip

\noindent {\it Proof of Theorem \ref{mainthmgr}:} It is now plain from Corollary \ref{6897245} and   Lemma  \ref{phi''copy}. \qed

\medskip

\section{Bloch-Kato Selmer Companion forms}\label{secbk}
In this section we study Bloch-Kato Selmer companion forms and establish our main result (Theorem \ref{mainthmgr2}). We shall begin by  comparing Greenberg and Bloch-Kato Selmer group. Recall from Theorem \ref{rhof}(2), $\lambda_{f_i}$ is the unramified character of $G_p$ with $\lambda_{f_i}(\mathrm{Frob}_p) =\alpha_p(f_i)$.

\medskip

\begin{proposition}\label{blochkatogreenbergcomparison}
\begin{enumerate} Let $i \in \{1,2\}$. We  assume all of the following hypotheses. 
\item Let $f_i$ be $p$-ordinary.
 \item The tame level $N$ of $f_i$ is square-free and $\forall \ell \in S$,  $ \text{cond}_\ell(\bar{\rho}_{f_i})=\ell$.
\item The condition $(\mathrm{C'_{i,\ell}})$, defined in equation \eqref{conduction-condition126}, is satisfied.
\item $H^1_\mathrm{Gr}(\Q_p, A_{f_i\chi(-j)})$ is $\pi$-divisible.
\item $\lambda_{f_i} \neq \pm 1$ and $\lambda_{f_i} \neq \pm\epsilon'_i$. Then,
\end{enumerate}
\begin{center}{$S_\mathrm{Gr}(A_{f_i\chi(-j)}[\pi^r]/\Q) \cong S_\mathrm{BK}(A_{f_i\chi(-j)}[\pi^r]/\Q).$}\end{center}
\end{proposition}

\proof It suffices to show, $H^1_\mathrm{BK}(\Q_q,  A_{{f_i}\chi(-j)}[\pi^r]) = H_\mathrm{Gr}^1(\Q_q, A_{{f_i}\chi(-j)}[\pi^r])$ for every $q \in \Sigma$.

\medskip

 By assumptions (1) and (2), it follows from Corollary \ref{cor3467} that  \begin{small}{$A_{{f_i}\chi(-j)}^{I_q}$}\end{small} is divisible for every  \begin{small}{$q \in \Sigma \setminus \{p\}$}\end{small}. Thus for such a $q$,  \begin{small}{$H_\mathrm{Gr}^1(\Q_q, A_{{f_i}\chi(-j)})=  H^1(G_q/{I_q}, A_{{f_i}\chi(-j)}^{I_q})$}\end{small} is divisible. Note that  \begin{small}{$H^1_\mathrm{BK}(\Q_q,  A_{{f_i}\chi(-j)})$}\end{small} is the maximal divisible subgroup of  \begin{small}{$H^1(G_q/{I_q}, A_{{f_i}\chi(-j)}^{I_q})$}\end{small} (see for example \cite[Proposition 4.2]{oc}). Hence we deduce for  every  \begin{small}{$q \in \Sigma \setminus \{p\}, H^1_\mathrm{BK}(\Q_q,  A_{{f_i}\chi(-j)}) $ \\ = $H_\mathrm{Gr}^1(\Q_q, A_{{f_i}\chi(-j)})$}\end{small} and consequently  \begin{small}{$H^1_\mathrm{BK}(\Q_q,  A_{{f_i}\chi(-j)}[\pi^r]) = H_\mathrm{Gr}^1(\Q_q, A_{{f_i}\chi(-j)}[\pi^r])$}\end{small} follows. 

\medskip

Next we consider the case of prime $p$. Note that in this case, the image of $H^1_g(\Q_p, V_{{f_i}\chi(-j)})$ in $H^1(\Q_p, A_{{f_i}\chi(-j)})$ is the maximum divisible subgroup of $H^1_\mathrm{Gr}(\Q_p, A_{{f_i}\chi(-j)})$ \cite[Proof of Proposition 4.2]{oc}. We first claim, $H^1_\mathrm{BK}(\Q_p, V_{{f_i}\chi(-j)}) = H^1_g(\Q_p, V_{{f_i}\chi(-j)})$. Assume this claim to be true at the moment.  Then from the assumption (3), we get that $H^1_\mathrm{BK}(\Q_p, A_{{f_i}\chi(-j)}) = H^1_\mathrm{Gr}(\Q_p, A_{{f_i}\chi(-j)})$. This in turn implies that 
$H^1_\mathrm{BK}(\Q_p,  A_{{f_i}\chi(-j)}[\pi^r])$ = $H_\mathrm{Gr}^1(\Q_p, A_{{f_i}\chi(-j)}[\pi^r])$.

The proposition follows from the above discussion once we establish the claim.  We calculate, $\text{dim}_{\Q_p} H^1_\mathrm{BK}(\Q_p, V_{{f_i}\chi(-j)}) - \text{dim}_{\Q_p}  H^1_g(\Q_p, V_{{f_i}\chi(-j)})$. By \cite[2.3.5, Page 35]{be},
$$
\text{dim}_{\Q_p} \frac{H_g^1(\Q_p, V_{{f_i}\chi(-j)})}{H_\mathrm{BK}^1(\Q_p, V_{{f_i}\chi(-j)} )} = \text{dim}_{\Q_p}  D_\text{crys}(V_{{f_i}\chi(-j)}^*(1))^{\phi=1}
$$
where $V^* = \text{Hom}_{\Q_p} (V, \Q_p)$

We shall first show  $D_\text{crys}({V''}_{{f_i}\chi(-j)}^*(1))^{\phi=1}  =0$. If $\chi$ is ramified then it follows from \begin{small}{(\cite[ \S7.2.4, \S7.2.5, Prop. 7.20]{fo})}\end{small} that ${V''}_{{f_i}\chi(-j)}^*(1)$ is not crystalline. Therefore  $ D_\text{crys}({V''}_{{f_i}\chi(-j)}^*(1))\\ = 0 $. Now suppose that $\chi$ is unramified. In this case ${V''}_{{f_i}\chi(-j)}^*(1)$ is crystalline. Then using  \cite[\S 4.2.3]{sc}) and the assumption $\lambda_{f_i} \neq \pm 1$ implies that $D_\text{crys}({V''}_{{f_i}\chi(-j)}^*(1))^{\phi=1} =0 $.

\medskip

Next, we show $D_\text{crys}({V'}_{{f_i}\chi(-j)}^*(1))^{\phi=1}=0$. The Galois group $G_p$ acts on ${V'}_{{f_i}\chi(-j)}^*(1)$ via $\lambda_{f_i}^{-1}\epsilon_i' \epsilon_{i,p} \chi \omega_p^{k-2-j}$. Now if $\epsilon_{i,p}\chi$ is ramified, then ${V'}_{{f_i}\chi(-j)}^*(1)$ is not crystalline and therefore  $D_\text{crys}({V'}_{{f_i}\chi(-j)}^*(1))=0$. 

\medskip

On the other hand, assume that $\epsilon_{i,p}\chi$ is unramified.  Moreover if $\chi$ is also unramified, we necessarily have $\epsilon_{i,p}=1$. 
Further  the assumption that $\lambda_{f_i} \neq \pm \epsilon_i'$ implies that 
$\lambda_{f_i}^{-1}\epsilon'_i  \chi \neq 1$. Then  from  \cite[\S 4.2.3]{sc}  
we have $D_\text{crys}({V''}_{{f_i}\chi(-j)}^*(1))^{\phi=1} =0 $. Finally,  if $\chi$ is ramified then we get $(\epsilon_{i,p}\chi)_{\mid_{I_p}}=1$. This shows that $\epsilon_{i,p}$ is a quadratic character. Therefore $\epsilon_{i,p}\chi$ is an unramified quadratic character. 
Now again from  \cite[\S 4.2.3]{sc})  and the assumption $\lambda_{f_i} \neq \pm \epsilon'_i$  we have that $D_\text{crys}({V''}_{{f_i}\chi(-j)}^*(1))^{\phi=1} =0 $. This completes the proof of the proposition. \qed

\medskip

\begin{rem}\label{miyakerem}
We note that the condition (4) in Proposition \ref{blochkatogreenbergcomparison} is satisfied if either weight $k >2$  or  if $p$ is a good prime for $f_i$  i.e. ${t_1}=t_2=0$    holds \cite[Theorem 4.6.17(3)]{mi}.

\end{rem}

\begin{lemma}\label{lemma2.13}
Let $f_i$ be $p$-ordinary. If  $\omega_p^{k-2-j} \epsilon_i \chi \neq\lambda_{f_i} $ (mod $\pi$) then $H^2(\Q_p,A'_{f_i\chi(-j))}[\pi])=0$ and in particular this implies that $H^1(\Q_p,A'_{f_i\chi(-j)})$ is $\pi$-divisible.
\end{lemma}
\begin{proof}
We write $\bar{\omega}_p : = \omega_p ~(\text{mod}~ \pi)$. By Tate duality, the Pontryagin dual of $H^2(\Q_p,A'_{f_i\chi(-j)}[\pi])$ is equal to $H^0(\Q_p,(A'_{f_i\chi(-j)}[\pi])^*(1))$, where $A'_{f_i\chi(-j)}[\pi])^*(1)$ is defined as $\mathrm{Hom}(A'_{f_i\chi(-j)}[\pi],\bar{\omega}_p)$. Thus it suffices to show $H^0(\Q_p,(A'_{f_i\chi(-j)}[\pi])^*(1))=0$ or equivalently $G_p$ acts non-trivially on $\mathrm{Hom}(A'_{f_i\chi(-j)}[\pi],\bar{\omega}_p)$. On the other hand, $\mathrm{Hom}(A'_{f_i\chi(-j)}[\pi],\bar{\omega}_p)$ is equal to the Pontryagin dual of $A'_{f_i\chi(-j-1)}[\pi]$. 
Thus $G_p$ acts non-trivially on $\mathrm{Hom}(A'_{f_i\chi(-j)}[\pi],\bar{\omega}_p)$  if and only if  $G_{p}$ acts non-trivially on $A'_{f_i\chi(-j-1)}[\pi]$. This action of  $G_{p}$   is via $\omega_p^{k-2-j} \epsilon_i \chi\lambda_{f_i}^{-1}$ (mod $\pi$). Thus using the given hypothesis, the first  of the lemma follows. 

\medskip

For the second part, there is an exact sequence via Kummer theory
\[  H^1(\Q_p,A'_{f_i\chi(-j)}) \lra H^1(\Q_p,A'_{f_i\chi(-j)}) \lra H^2(\Q_p,A'_{f_i\chi(-j))}[\pi]) \]
where the first map is given  by  multiplication by $\pi $. Hence if $H^2(\Q_p,A'_{f_i\chi(-j))}[\pi])=0$ then  $H^1(\Q_p,A'_{f_i\chi(-j)})$ is   $\pi$-divisible.
\end{proof}

\begin{lemma}\label{divpncase}
Let $i=1,2$ and $f_i$ be $p$-ordinary. Assume the following conditions. 
\begin{enumerate} 

\item $H^2(\Q_p, A'_{f_i\chi(-j)}[\pi]) =0$.
\item $A''^{I_p}_{f_i\chi(-j)}$ is $\pi$-divisible.
\end{enumerate} 
Then  $H^1_\mathrm{Gr}(\Q_p, A_{f_i\chi(-j)})$ is $\pi$-divisible.
\end{lemma}
\begin{proof}
Consider  the natural restriction map $g: H^1(\Q_p, A''_{f_i\chi(-j)}) \lra H^1(I_p, A''_{f_i\chi(-j)})$. Then $\text{Ker}(g) \cong H^1(G_p/I_p, A''^{I_p}_{f_i\chi(-j)}) \cong H^1(<\text{Frob}_p>, A''^{I_p}_{f_i\chi(-j)})$, is divisible by assumption (2).
On the other hand, the exact sequence $0 \lra A'_{f_i\chi(-j)} \lra A_{f_i\chi(-j)} \lra A^{''}_{f_i\chi(-j)} \lra 0$ induces 
the  natural maps on cohomology  $$  H^1(\Q_p, A'_{f_i\chi(-j)})  \stackrel{\psi}{\lra} H^1(\Q_p, A_{f_i\chi(-j)})  \stackrel{f}{\lra} H^1(\Q_p, A''_{f_i\chi(-j)})$$ Then $f$ is surjective by assumption (1).  Also  $\text{Ker}(f) \cong \text{Img}(\psi)$ is divisible by our assumption (1) together with  Lemma \ref{lemma2.13}. Now we consider the exact sequence 
\[ 0 \lra Ker(f) \lra Ker (g\circ f)= H^1_\mathrm{Gr}(\Q_p, A_{f_i\chi(-j)}) \lra Ker(g) \lra Coker(f)=0   \]
 The divisibility of $H^1_\mathrm{Gr}(\Q_p, A_{f_i\chi(-j)})$ follows from the divisibility of  $\text{Ker}(f)$ and $\text{Ker}(g)$.
\end{proof}
\begin{lemma}\label{hgfgrsfdjl,kjf}
 If either one of the following conditions hold
\begin{enumerate}
\item[(i)] $p >2k-3$ and $\omega_p^{k-2-j}\epsilon_{i,p}$ (mod $\pi$) is not a quadratic character.
\item[(ii)]  $p >2k-3$ and $a_p(f_i) \neq {\pm }\epsilon_i'(\text{Frob}_p)$ (mod $\pi$),
\end{enumerate}
then the assumptions of Lemma \ref{divpncase} hold and  $H^1_\mathrm{Gr}(\Q_p, A_{f_i\chi(-j)})$ is $\pi$-divisible.

In particular, if $t_1=t_2=0$, then   if $p> 2k -3 $ and either (i)  $j\neq k-2$ or  (ii) $a_p(f_i)\neq \pm \epsilon_i(\text{Frob}_p)$ (mod $\pi$) holds, then   $H^1_\mathrm{Gr}(\Q_p, A_{f_i\chi(-j)})$ is $\pi$-divisible.
\end{lemma}

\proof Recall $\lambda_{f_i}(\text{Frob}_p) = \alpha_p(f_i) \equiv a_p(f_i) $ (mod  $\pi)$. If $p>2k-3$ then $A''^{I_p}_{f_i\chi(-j)}$ is $\pi$-divisible (by Corollary \ref{78914121881} and Remark \ref{div}). In addition, if $\omega_p^{k-2-j}\epsilon_{i,p}$ (mod $\pi$) is not a quadratic character then $(\omega_p^{k-2-j}\epsilon_{i,p})_{\mid_{I_p}} \neq \pm1$ (mod $\pi$) as $\omega_p$ and $\epsilon_{i,p}$ are determined by their values on $I_p$. Hence $(\omega_p^{k-2-j}\epsilon_{i,p}\chi)_{\mid_{I_p}} \neq 1$ (mod $\pi$)  for any quadratic character $\chi$.  As $\epsilon'_i$ and $\lambda_{f_i}$ are unramified at $p$, we deduce that $\omega_p^{k-2-j}\epsilon_{i,p}\chi \neq {\epsilon'_i}^{-1} \lambda_{f_i}$ (mod $\pi$) or equivalently $\omega_p^{k-2-j}\epsilon_i\chi \neq  \lambda_{f_i}$ (mod $\pi$). Using this, we get  $H^2(\Q_p, A'_{f_i\chi(-j)}[\pi]) =0$ by Lemma \ref{lemma2.13}. Consequently, $H^1_\mathrm{Gr}(\Q_p, A_{f_i\chi(-j)})$ is $\pi$-divisible. 

\medskip

 In the second case,  assume   $p>2k-3$,  $a_p(f_i) \neq {\pm }\epsilon_i'(\text{Frob}_p)$ (mod $\pi$) and $\omega_p^{k-2-j}\epsilon_{i,p}$ (mod $\pi$) is a quadratic character.  Then $\omega_p^{k-2-j}\epsilon_{i,p}\chi$ (mod $\pi$) is quadratic   and ${\epsilon_i'}^{-1}\lambda_{f_i}$ is not a quadratic character.    Therefore $\omega_p^{k-2-j}\epsilon_{i,p}\chi \neq {\epsilon'_i}^{-1}\lambda_{f_i}$ (mod $\pi$). Thus again by Lemma \ref{lemma2.13}, $H^2(\Q_p, A'_{f_i\chi(-j)}[\pi]) =0$ and $H^1_\mathrm{Gr}(\Q_p, A_{f_i\chi(-j)})$ is $\pi$-divisible in this case as well.

\begin{lemma}\label{combiningwithdswandourresult}
Recall $f_i \in S_{k}(\Gamma_0(Np^{t_i}), \epsilon_i)$ be $p$-ordinary, where $i=1,2$. Let us assume $k\neq 3$, $p>2k-3$. Also assume  $t_1=t_2=0$ i.e. $f_i \in S_{k}(\Gamma_0(N), \epsilon_i)$ with $p \nmid N$. Consider the isomorphism $ H^1(\Q_p, A_{f_1\chi(-j)}[\pi^r]) \stackrel{[\phi^r]}{\lra} H^1(\Q_p, A_{f_2\chi(-j)}[\pi^r])$ induced from the $G_\Q$ linear isomorphism $\phi: A_{f_1\chi(-j)}[\pi^r] \lra A_{f_2\chi(-j)}[\pi^r]$ with $j=k-2$. Then for $j=k-2$, $[\phi]$ induces an isomorphism $H^1_\mathrm{BK}(\Q_p, A_{f_1\chi(-j)}[\pi^r]) \stackrel{[\phi]}{\lra} H^1_\mathrm{BK}(\Q_p, A_{f_2\chi(-j)}[\pi^r])$ for every quadratic character $\chi$ of $G_\Q$.
\end{lemma}
\proof We first consider the case when $\chi$ is a ramified character at $p$. Since $j=k-2$ and $\epsilon_i$, $\lambda_{f_i} $ are unramified at $p$, we deduce $ \epsilon_i \chi \neq\lambda_{f_i} $ (mod $\pi$) and hence by Lemma \ref{lemma2.13}, $H^2(\Q_p,A'_{f_i\chi(-j))}[\pi^r])=0$ for $i=1,2$. Also note that $(A^{''}_{f_{i}\chi(-j)})^{I_p}$ is divisible by Remark \ref{div} and thus by Lemma \ref{divpncase},
  $H^1_\mathrm{Gr}(\Q_p, A_{f_i\chi(-j)})$ is $\pi$-divisible. Further, by Remark \ref{miyakerem} and from the proof of Proposition \ref{blochkatogreenbergcomparison}, we obtain  $H^1_\mathrm{BK}(\Q_p, A_{f_i\chi(-j)}[\pi^r]) \cong  H^1_\mathrm{Gr}(\Q_p, A_{f_i\chi(-j)}[\pi^r]) $ for $i=1,2$. Finally, applying Proposition \ref{ipandip'}, we deduce the lemma in this case. 
  
  \medskip

Next we consider the case when $\chi$ is  unramified  at $p$. In this case as $t_i=0$, $V_{f_i\chi}$ is crystalline at $p$ for $i=1,2$. Then for any $j \neq \frac{k-1}{2}$, under the $G_\Q$ linear isomorphism $\phi: A_{f_1\chi(-j)}[\pi^r] \lra A_{f_2\chi(-j)}[\pi^r]$, it is shown in \cite[Theorem 6.1, Case (3), Page 10]{dsw} that  $[\phi]$ induces an isomorphism $H^1_\mathrm{BK}(\Q_p, A_{f_1\chi(-j)}[\pi^r]) \stackrel{[\phi]}{\lra} H^1_\mathrm{BK}(\Q_p, A_{f_2\chi(-j)}[\pi^r])$. Note that in the proof of \cite[Theorem 6.1, Case (3), Page 10]{dsw}  the case $r=1$ is covered; however from their proof we can see that the  result hold for a general $r$ as well. Now as $k \neq 3$, $k-2 \neq \frac{k-1}{2}$ and  the assertion in the lemma follows from the above mentioned result of \cite{dsw}. 
\begin{rem}
In fact, the proof of \cite[Theorem 6.1, Case (3), Page 10]{dsw} identifying $H^1_\mathrm{BK}(\Q_p, A_{f_1\chi(-j)}[\pi^r])$ with $H^1_\mathrm{BK}(\Q_p, A_{f_2\chi(-j)}[\pi^r])$ works for the pair $V_{f_1\chi(-j)}$ and $V_{f_2\chi(-j)}$ in both the cases when $f_1$ and $f_2$ are either ordinary or non-ordinary  at $p$; however they require $V_{f_i\chi(-j)}$ to be crystalline for $i=1,2$ as well as $j\neq \frac{k-1}{2}$. Thus the conclusion of Lemma \ref{combiningwithdswandourresult} can not be deduced only from the results of \cite{dsw}.
\end{rem} 

\medskip

Recall, if $p$ does not divide level of $h$ and $v_p(a_p(h)) \neq 0$, then for any  $K$ with $\Q_p \subset K \subset \Q_p(\mu_{p^\infty})$ we have defined the signed $i$ Selmer  local condition $H^1_i(K, A_{h(-j)}[\pi^r]) $ in \S \ref{sec2}.

\medskip

\begin{lemma}\label{signedselmercongruence70897}
Let $p \nmid N$ and for $i \in \{1, 2\}$, let $f_i \in S_k(\Gamma_0(N),\epsilon_i)$ be non-ordinary at $p$  with $p \geq k$. Let $K$ be any field such that $\Q_p(\mu_p) \subset K \subset \Q_p(\mu_{p^\infty})$. Consider the isomorphism $ H^1(K, A_{f_1(-j)}[\pi^r]) \stackrel{[\phi]}{\lra} H^1(K, A_{f_2(-j)}[\pi^r])$ induced from the $G_\Q$ linear isomorphism $\phi=\phi^r: A_{f_1(-j)}[\pi^r] \lra A_{f_2(-j)}[\pi^r]$ where $j \in \Z$. Then $[\phi]$ induces an isomorphism  $H^1_1(K, A_{f_1(-j)}[\pi^r]) \cong H^1_1(K, A_{f_2(-j)}[\pi^r])$. In particular, for $0 \leq j \leq k-2$, we have a canonical identification induced by $\phi$,
$$H^1_1(\Q_p(\mu_p), A_{f_1(-j)}[\pi^r]) \cong H^1_1(\Q_p(\mu_p), A_{f_2(-j)}[\pi^r]).$$
\end{lemma}
 \proof First of all, by an argument entirely similar to \cite[Lemma 4.4]{lh},  for every $j$, we have a canonical isomorphism induced by $\phi$
 \begin{equation}\label{afdkajFKLAUR}
 H^1_1(\Q_p(\mu_{p^{\infty}}), A_{f_1(-j)}[\pi^r]) \cong H^1_1(\Q_p(\mu_{p^{\infty}}), A_{f_2(-j)}[\pi^r]).
 \end{equation}
 Under the assumption $p\geq k$, it follows by \cite[Lemma 4.4]{lei} that 
 $$A_{f_i(-j)}^{G_{\Q_p(\mu_{p^\infty})}} =0.$$ 
 This is used in \cite[Lemma 4.3]{lh} to deduce  for $i=1,2$ and any $j$
 \begin{equation}\label{zsdkjalkRUL}
 H^1_1(\Q_p(\mu_{p^\infty}), A_{f_i(-j)}[\pi^r]) \cong H^1_1(\Q_p(\mu_{p^\infty}), A_{f_i(-j)})[\pi^r].
 \end{equation}
Further by \cite[Remark 2.5]{lh} we see that 
\begin{equation}\label{jdkadiqIPOap}
H^1_1(\Q_p(\mu_{p^\infty}), A_{f_i(-j)})^{\mathrm{Gal}(\Q_p(\mu_{p^\infty})/K)} \cong H^1_1(K, A_{f_i(-j)}).
\end{equation} 
From \eqref{zsdkjalkRUL} and \eqref{jdkadiqIPOap}, we deduce
\begin{equation}\label{lawkdpq][D[]QOE[10}
H^1_1(K, A_{f_i(-j)})[\pi^r] \cong H^1_1(\Q_p(\mu_{p^\infty}), A_{f_i(-j)}[\pi^r])^{\mathrm{Gal}(\Q_p(\mu_{p^\infty})/K)} 
\end{equation}
On the other hand, from the definition of $H^1_1(K, A_{f_i(-j)}[\pi^r])$ and the  canonical identification of $H^1(K, A_{f_i(-j)}[\pi^r]) \cong H^1(K, A_{f_i(-j)})[\pi^r]  $ induced via the inclusion map, we have 
\begin{equation}\label{kj4t48rjalsddj}
H^1_1(K, A_{f_i(-j)}[\pi^r])\cong H^1_1(K, A_{f_i(-j)})[\pi^r]
\end{equation}
From \eqref{lawkdpq][D[]QOE[10} and \eqref{kj4t48rjalsddj},   for every  $j$, we have  a canonical  isomorphism 
\begin{equation}\label{JKHSDYIOQ}
H^1_1(K, A_{f_1(-j)}[\pi^r]) \cong H^1_1(K, A_{f_2(-j)}[\pi^r])
\end{equation}
 which is induced from the isomorphism  $ H^1(K, A_{f_1(-j)}[\pi^r]) \stackrel{[\phi]}{\lra} H^1(K, A_{f_2(-j)}[\pi^r])$ coming from the $G_\Q$ linear isomorphism $\phi$. \qed

\begin{proposition}\label{comparison-ss-bk-s1-s2}
Let $p \nmid N$ and for $i \in \{1, 2\}$, let $f_i \in S_k(\Gamma_0(N),\epsilon_i)$ be non-ordinary at $p$ with $p \geq k$. Then for every $0 \leq j \leq k-2$,
\begin{equation}\label{bloch-kato-selmer-greenberg-signed}
H^1_\mathrm{BK}(\Q_p(\mu_p), A_{f_i(-j)})=H^1_{1}(\Q_p(\mu_p), A_{f_i(-j)}).
\end{equation} 
\end{proposition}
\proof The main idea is that Bloch-Kato condition  at $p$ is given by the kernel of the Perrin-Riou dual exponential map and local condition at $p$ for the signed $i$ Selmer group is given by the kernel of the Colman$_i$ map. One then deduces the result by looking at the relation between the dual exponential map and the Colman map. 

\medskip

 It is shown in \cite[Prop. 2.14, Remark 2.15]{lh} that for every $0 \leq j \leq k-2$,
\begin{equation}\label{akjdfhaahdad}
H^1_\mathrm{BK}(\Q_p, A_{f_i(-j)})=H^1_{1}(\Q_p, A_{f_i(-j)}) = H^1_2(\Q_p, A_{f_i(-j)})
\end{equation}
Here we remark that it follows from \cite[Remark 5.8]{llzant} that the condition in \cite[Remark 2.15]{lh} is always satisfied. We thank Antonio Lei for explaining this. We will use the results in \cite[\S 3.3]{llz}, to extend \eqref{akjdfhaahdad} to $\Q_p(\mu_p)$. 

We define $H^1_{BK}(\Q_p(\mu_p), T_{h(-j)})$ to be the preimage of $H^1_{BK}(\Q_p(\mu_p), V_{h(-j)})$ under the natural map induced by the inclusion $T_{h(-j)} \rightarrow V_{h(-j)}$. Let $z \in H^1_\mathrm{Iw}(\Q_p, T_{h})$ and we denote by $z_j$ its image in $H^1(\Q_p(\mu_p), T_{h(-j)})$ under the natural composite map 
$H^1_\mathrm{Iw}(\Q_p, T_{h}) \rightarrow H^1(\Q_p(\mu_p), T_{h}) \rightarrow H^1(\Q_p(\mu_p), T_{h(-j)})$. We choose two distinct $u_1$ and $u_2$ in $\Z_p^\times$ in \cite[Prop. 3.11]{llz} and use it in \cite[Equation (3.6), Page 836]{llz} for $n=1$ to deduce for $h \in \{f_1, f_2\}$, 
\begin{small}{
$$z_j\in H^1_{BK}(\Q_p(\mu_p), T_{h(-j)}) \Leftrightarrow z_j \in \text{Pr}_{\Q_p(\mu_p)}(\text{Ker}(\text{Col}_{h,1})\otimes \omega_p^{-j}) \Leftrightarrow z_j \in \text{Pr}_{\Q_p(\mu_p)}(\text{Ker}(\text{Col}_{h,2})\otimes \omega_p^{-j}),$$
}\end{small}
 where $0 \leq j \leq k-2$. Then it follows from the definitions of Bloch-Kato and signed Selmer condition that for $i=1,2$ 
 \begin{equation}\label{akjdfhaahdadmuptlevel}
H^1_{BK}(\Q_p(\mu_p), T_{f_i(-j)})=H^1_{1}(\Q_p(\mu_p), T_{f_i(-j)}) = H^1_2(\Q_p(\mu_p), T_{f_i(-j)})
\end{equation}
Consequently, we obtain
 
\begin{equation}\label{akjdfhaahdadmup}
H^1_{BK}(\Q_p(\mu_p), A_{f_i(-j)})=H^1_{1}(\Q_p(\mu_p), A_{f_i(-j)}) = H^1_2(\Q_p(\mu_p), A_{f_i(-j)})
\end{equation}
Then from the definition of $\pi^r$-Selmer condition given in \eqref{sel-mod-r-general}, it follows that for $i=1,2$,
\begin{equation}\label{bloch-kato-selmer-greenberg-signed1243}
H^1_{BK}(\Q_p(\mu_p), A_{f_i(-j)}[\pi^r])=H^1_{1}(\Q_p(\mu_p), A_{f_i(-j)}[\pi^r])=H^1_{2}(\Q_p(\mu_p), A_{f_i(-j)}[\pi^r]).
\end{equation} 
In particular, \eqref{bloch-kato-selmer-greenberg-signed} holds. \qed

\medskip

Now, we are ready to prove our main result of this article.

\medskip

\begin{theorem}\label{mainthmgr2}
Let $p$ be an odd prime and for $i=1,2,$ let $f_i$ be a  normalized cuspidal Hecke eigenform in $S_k(\Gamma_0(Np^{t_i}), \epsilon_i)$, where $(N,p) =1$, $k \geq 2$ and $t_i \in \N \cup \{0\}$. 
Let $r \in \N$ and $\phi : A_{f_1}[\pi^r] \lra A_{f_2}[\pi^r]$ be a $G_\Q$ linear isomorphism. We  assume the following:
\begin{enumerate}

\item $N$ is square-free and $\forall \ell \in S$,  $ \text{cond}_\ell(\bar{\rho}_{f_i})=\ell$ for $i=1,2$.
\item The condition $(\mathrm{C'_{i,\ell}})$, defined in equation \eqref{conduction-condition126}, is satisfied for $i=1,2$.

\item  Either [$p$-ss] or [$p$-ord] holds.

\begin{enumerate}
\item[{[$p$-ss]}] $p\geq k$, $f_1$and $f_2$ are non-ordinary at $p$,  and $t_1=t_2=0$.
\item[{[$p$-ord]}] $p > 2k-3$, $f_1$ and $f_2$ are ordinary at $p$,  and either [$p$-good]  or [$p$-general] holds.
\begin{enumerate} 

\item[{[$p$-good]}] $t_1=t_2=0$ and $k \neq 3$.
\item[{[$p$-general]}] All of (A), (B) and (C) are satisfied.
\begin{enumerate}
\item $a_p(f_i) \neq \pm \epsilon_i'(\text{Frob}_p)$ (mod $\pi$). 
\item  $\omega_p^{k-1}\epsilon_{i,p} \neq 1$ (mod $\pi$) for $i=1,2$. 
\item If both $k=2 $ and $t_i >0$ holds, then in addition assume $\lambda_{f_i} \neq \pm 1.$
\end{enumerate}
\end{enumerate}
\end{enumerate}

\end{enumerate}
Then  for every quadratic character $\chi $ of $G_\Q$ and for every fixed $j$ with $0 \leq j \leq k-2$, we have an isomorphism 
of the $\pi^r$- Bloch-Kato Selmer groups 
$$S_\mathrm{BK}(A_{f_1\chi (-j)}[\pi^{r}]/\Q) \cong S_\mathrm{BK}( A_{f_2\chi (-j)}[\pi^{r}]/\Q). $$
\end{theorem}
\proof  First we consider the case (1), (2) and [$p$-ord] are satisfied.  Note that by Theorem \ref{rhof}(2) and Lemma  \ref{phi''copy}, $a_p(f_1) \equiv a_p(f_2)$ (mod $\pi^r$) and hence  $a_p(f_1) \equiv a_p(f_2)$ (mod $\pi$) also. 
Now we consider a subcase where [$p$-ord] is satisfied via [$p$-general]. Then  by hypothesis  (A)  i.e. $a_p(f_i) \neq \pm \epsilon_i'(\text{Frob}_p)$ (mod $\pi$)    together with Lemma \ref{hgfgrsfdjl,kjf}(ii),  we get $H^1_\mathrm{Gr}(\Q_p, A_{f_i\chi(-j)})$ is $\pi$-divisible for $i=1,2$. 
Therefore using conditions (1), (2) and   $\lambda_{f_i} \neq \pm 1$ (if necessary), Proposition \ref{blochkatogreenbergcomparison} and Remark \ref{miyakerem}, we deduce for each fixed $j$ with $0 \leq j \leq k-2$ and for every $\chi$,
\begin{equation}\label{bkgrbkgrbkgr}
 S_\mathrm{BK}( A_{f_i\chi (-j)}[\pi^{r}]/\Q) \cong S_\mathrm{Gr}( A_{f_i\chi(-j)}[\pi^{r}]/\Q)  .
 \end{equation}
Further given the hypothesis (B) i.e.  $\omega_p^{k-1}\epsilon_{i,p} \neq 1$ (mod $\pi$) for $i=1,2$, we  apply  \eqref{bkgrbkgrbkgr} in Theorem \ref{mainthmgr} to deduce  Theorem \ref{mainthmgr2} in this case.

\medskip

On the other hand if [$p$-ord] is satisfied via [$p$-good], then by Lemma \ref{hgfgrsfdjl,kjf}(i),  we get $H^1_\mathrm{Gr}(\Q_p, A_{f_i\chi(-j)})$ is $\pi$-divisible for $i=1,2$ as long as $j \neq k-2$.  Therefore once again using conditions (1), (2),   Proposition \ref{blochkatogreenbergcomparison} and Remark \ref{miyakerem}, we deduce for $j \neq k-2$ and for every $\chi$, the isomorphism in \eqref{bkgrbkgrbkgr} between Greenberg and Bloch-Kato Selmer group continue to hold. Also as $p> 2k-3$ and $t_1=t_2=0$,  $\omega_p^{k-1}  \neq 1$ (mod $\pi$) and  $\epsilon_{i,p}=1$ for $i=1,2$. Hence as in previous paragraph, we use \eqref{bkgrbkgrbkgr} in Theorem  \ref{mainthmgr} to obtain  $0\leq j \leq k-3$ case of Theorem \ref{mainthmgr2}. On the other hand, if $j=k-2$ then as $k\neq 3$ by our assumption, we apply Lemma \ref{combiningwithdswandourresult} to directly obtain the canonical isomorphism  induced from $\phi$, 
\begin{equation}\label{jdhduuiepopo}
H^1_\mathrm{BK}(\Q_p, A_{f_1\chi(-j)}[\pi^r]) \stackrel{[\phi]}{\lra} H^1_\mathrm{BK}(\Q_p, A_{f_2\chi(-j)}[\pi^r])
\end{equation}
 for every quadratic character $\chi$ of $G_\Q$.  Moreover as explained in the proof of Proposition \ref {blochkatogreenbergcomparison}, for every  \begin{small}{$q \in \Sigma \setminus \{p\}$}\end{small}, by assumptions (1), (2) and Corollary \ref{cor3467} we have  $H^1_\mathrm{BK}(\Q_q,  A_{{f_i}\chi(-j)}[\pi^r]) = H_\mathrm{Gr}^1(\Q_q, A_{{f_i}\chi(-j)}[\pi^r]) $ and thus for $i=1,2$ and for every $j$ such that $0\leq j \leq k-2$, we can canonically identify, 
 \begin{equation}\label{xajhgcbjaha}
\frac{H^1(\Q_q,  A_{{f_i}\chi(-j)}[\pi^r])}{H^1_\mathrm{BK}(\Q_q,  A_{{f_i}\chi(-j)}[\pi^r])} \cong H^1(I_q,  A_{{f_i}\chi(-j)}[\pi^r]), \quad q \in \Sigma \setminus \{p\}
\end{equation}
From \eqref{jdhduuiepopo} and \eqref{xajhgcbjaha}, by the definition of $\pi^r$ Bloch-Kato Selmer group, the remaining  $j=k-2$ case of Theorem \ref{mainthmgr2} in the [$p$-good] case is obtained. This completes the proof of Theorem \ref{mainthmgr2} when assumptions (1), (2) and [$p$-ord] are satisfied.

 \medskip
 
 Next we consider the case when assumptions  (1), (2) and [$p$-ss] are satisfied. Then the definition of the local factor for $S_\mathrm{BK}(A_{f_1\chi (-j)}[\pi^{r}]/\Q)$ at a prime $q \in \Sigma \setminus \{p\} $ is the same as in the [$p$-ord] case. Hence using the same argument as in the [$p$-ord] case, we obtain the identification of \eqref{xajhgcbjaha} for every quadratic character $\chi$ of $G_\Q$ and every $0 \leq j \leq k-2$ thanks to our assumptions (1) and (2). Thus to prove Theorem \ref{mainthmgr2}, it suffices  to establish a canonical identification induced by $\phi$ for every $\chi$ and $0 \leq j \leq k-2$, analogues to \eqref{jdhduuiepopo} under the assumption [$p$-ss].
 
 \medskip
 
 Now as $p \nmid N$, $t_1=t_2=0$ and $p \geq k$, using Lemma  \ref{signedselmercongruence70897} and Proposition \ref{comparison-ss-bk-s1-s2}, we get a canonical isomorphism 
\begin{equation}\label{bloch-kato-selmer-greenberg-signed1243adj}
H^1_\mathrm{BK}(\Q_p(\mu_p), A_{f_1(-j)}[\pi^r]) \cong H^1_\mathrm{BK}(\Q_p(\mu_p), A_{f_2(-j)}[\pi^r]).
\end{equation} 
for every $0 \leq j \leq k-2$ which is induced by $\phi$. Now let $\chi$ be a quadratic character. Following the notation of \S \ref{sec0}, we can write $\chi=\chi_p\chi'$ where $\chi_p$ is quadratic character whose conductor is a power of $p$ and $\chi'$ is unramified at $p$. Note that $f_i\otimes \chi'$ is good at $p$ i.e. the level of $f_i\otimes \chi'$ is coprime to $p$. Thus from \eqref{bloch-kato-selmer-greenberg-signed1243adj}, we still have a canonical isomorphism
\begin{equation}\label{bloch-kato-selmer-greenberg-signed1243adjadhjk}
H^1_\mathrm{BK}(\Q_p(\mu_p), A_{f_1\chi'(-j)}[\pi^r]) \cong H^1_\mathrm{BK}(\Q_p(\mu_p), A_{f_2\chi'(-j)}[\pi^r]).
\end{equation} 
induced by $\phi$ for every $j$. Notice that $\chi_p$ is the unique quadratic character of $\Delta= \mathrm{Gal}(\Q_p(\mu_p)/\Q_p)$, whence $H^1_\mathrm{BK}(\Q_p(\mu_p), A_{f_i\chi'(-j)}[\pi^r])\otimes{\chi_p} \cong  H^1_\mathrm{BK}(\Q_p(\mu_p), A_{f_i\chi'\chi_p(-j)}[\pi^r])$ for $i=1,2$. Thus we have,
\begin{equation}\label{bloch-kato-selmer-greenberg-signed1243adjadhjkdskadj}
H^1_\mathrm{BK}(\Q_p(\mu_p), A_{f_1\chi(-j)}[\pi^r]) \cong H^1_\mathrm{BK}(\Q_p(\mu_p), A_{f_2\chi(-j)}[\pi^r]).
\end{equation} 
Now we take invariance by $\Delta$ in  \eqref{bloch-kato-selmer-greenberg-signed1243adjadhjkdskadj}. Then using inflation-restriction sequence and  the fact that order of $\Delta$ is co-prime to $p$, we deduce for quadratic every character $\chi$ of $G_\Q$ and $0\leq j \leq k-2$, a canonical isomorphism
\begin{equation}\label{1243adjadhjkdskadjsmmdk}
H^1_\mathrm{BK}(\Q_p, A_{f_1\chi(-j)}[\pi^r]) \cong H^1_\mathrm{BK}(\Q_p, A_{f_2\chi(-j)}[\pi^r]).
\end{equation} 
induced by $\phi$. This completes the comparison of local $\pi^r$ Bloch-Kato condition at $p$ of $f_1$ and $f_2$, under the condition [$p$-ss]. This completes the proof of the Theorem \ref{mainthmgr2} under the hypotheses (1), (2) and [$p$-ss], as required. \qed

\begin{corollary}\label{dfbwklwfi3fmc}
From the proof of Theorem \ref{mainthmgr2}, we can identify the Bloch-Kato Selmer group with the signed Selmer group for $K \in \{\Q, \Q(\mu_p)\}$. More precisely, let $h \in S_k(\Gamma_0(N),\psi)$ be a non-ordinary at $p \geq k$ with $(p, N)=1$ and $p \geq k$. Then for $K \in \{\Q, \Q(\mu_p)\}$,
$$S_\mathrm{BK}(A_{h (-j)}/K) \cong S_1(A_{h(-j)}/K) \cong S_2(A_{h (-j)}/K)$$
where $S_i(A_{h (-j)}/K)$ is the signed $i$ Selmer group  of $h$ over $K$ and $0 \leq j \leq k-2$. 
\end{corollary}

\begin{corollary}\label{isogreenbloch-katopir}
Let $h \in S_k(\Gamma_0(N), \epsilon)$ be a $p$-ordinary newform with $(p, N)=1$ and $p>2k-3$. Assume  conditions (1), (2)  and [$p$-general]  of  Theorem \ref{mainthmgr2} holds for $h$.  Then for  $0 \leq j \leq k-2$, and for every quadratic character $\chi$, 
$$S_\mathrm{BK}( A_{h\chi (-j)}[\pi^{r}]/\Q) \cong S_\mathrm{Gr}( A_{h\chi(-j)}[\pi^{r}]/\Q).\qed$$
\end{corollary}
\begin{rem}\label{dfbn3hr39pr9r090rjn}
When $f$ is $p$-ordinary, it can be checked that $H^1_\mathrm{BK}(\Q_p, A_{f\chi(-j)}[\pi^r]) = {i^\ast_r}^{-1}\Big(\psi\big(H^1(\Q_p, A'_{f\chi(-j)})_\mathrm{div}\big)\Big),$ where $H^1(\Q_p, A'_{f\chi(-j)})  \stackrel{\psi}{\lra} H^1(\Q_p, A_{f\chi(-j)}) $ is the natural map induced by inclusion and $i^\ast_r: H^1(\Q_p, A_{f\chi(-j)}[\pi^r])$   $\lra H^1(\Q_p, A_{f\chi(-j)})$ is induced by the  Kummer map. 
Further when  $f$  corresponds to  an elliptic curve $E$  over $\Q$, then the  condition $a_p(f) \neq \pm \epsilon'_i(\text{Frob}_p)$ (mod $\pi$) is $a_p(f) \neq \pm 1 $ (mod $p$). In particular, $a_p(f)  \neq 1$ (mod $p$)  is precisely the condition $p \nmid \# \tilde{E}(\mathbb F_p)$. Such a prime is called a  non-anomalous prime (cf. \cite{gr}). \
\end{rem}

\medskip

 We  now extend the notion of Selmer companion forms to cupsforms of two different weights. 

\begin{defn}\label{74892747947928340}
Let $p \nmid N$ and  $f_i \in S_{k_i}(\Gamma_0(Np^{t_i}), \epsilon_i)$ be a  normalized cuspidal eigenform for $i=1,2$.  Then $f_1$ and $f_2$ are $\pi^r$  (Bloch-Kato) Selmer companion if for each critical twist $j$ with $0\leq j \leq \text{min}\{k_1-2, k_2-2\}$ and for every quadratic character $\chi$ of $G_\Q$,  $$S_\mathrm{BK}(A_{f_1\chi(-j)}[\pi^{r}]/\Q) \cong S_\mathrm{BK}( A_{f_2\chi(-j)}[\pi^{r}]/\Q).$$
\end{defn}

\begin{corollary}\label{mainthmgr3}
Let $p$ be odd  and $f_i \in S_{k_i}(\Gamma_0(Np^{t_i}), \epsilon_i)$ be a normalized cuspidal eigenform  with $p \nmid N$, $k_i \geq 2$ and $t_i \in \N \cup \{0\}$, $i=1,2$.  Let  $\phi : A_{f_1}[\pi^r] \lra A_{f_2}[\pi^r]$ be a $G_\Q$ linear isomorphism. We  assume the following:
\begin{enumerate}
\item $N$ is square-free and $\forall \ell \in S$,  $ \text{cond}_\ell(\bar{\rho}_{f_i})=\ell$ for $i=1,2$.
\item The condition $(\mathrm{C'_{i,\ell}})$, defined in equation \eqref{conduction-condition126}, is satisfied for $i=1,2$.
\item \begin{small}{$p>2 \text{max}\{k_1,k_2\}-3$}\end{small}, $f_1$ and $f_2$ are ordinary at $p$. 

\item $a_p(f_i) \neq \pm \epsilon_i'(\text{Frob}_p)$ (mod $\pi$) and  $\omega_p^{k_i-1}\epsilon_{i,p} \neq 1$ (mod $\pi$) for $i=1,2$. 
\item If $k_i=2 $ and $t_i>0$, then in addition assume $\lambda_{f_i} \neq \pm 1.$
\end{enumerate}

Then  for every quadratic character $\chi $ of $G_\Q$ and for every fixed $j$ with $0 \leq j \leq \text{min } \{k_1-2,k_2-2 \}$, we have an isomorphism 
of the $\pi^r$- Bloch-Kato Selmer groups 
$$S_\mathrm{BK}(A_{f_1\chi (-j)}[\pi^{r}]/\Q) \cong S_\mathrm{BK}( A_{f_2\chi (-j)}[\pi^{r}]/\Q). $$

\proof The proof is similar to the proof of  Theorem \ref{mainthmgr2}  when conditions (1), (2), and [$p$-ord] (via [$p$-general]) conditions  in Theorem \ref{mainthmgr2} are satisfied. Hence the proof is omitted.  \qed

\end{corollary}

\section{The cyclotomic case}\label{seccyc}
Let $\Q_\cyc$ be the cyclotomic $\Z_p$ extension of $\Q$.  For a prime $q \in \Q$, let $q_\infty$ be a prime in $\Q_\cyc$ dividing $q$. Let $\Q_{\cyc,q_\infty}$ denote the completion of $\Q_{\cyc}$ at $q_\infty$. Also $I_{q_\infty}$ and $G_{\Q_{\cyc,q_\infty}}$ will respectively denote the inertia and decomposition subgroup of $G_{\Q_\cyc}$ at $q_\infty$. Let $\Sigma_\infty$ be the set of all primes of  $ \Q_\cyc$ lying above the  primes of $\Sigma$. Let $h \in S_k(\Gamma_0(N), \epsilon)$ be a $p$-ordinary. For $\dagger \in \{ \mathrm{Gr},  \mathrm{BK}\}$, we define
$$
S_\dagger(A_{h\chi(-j)}[\pi^r]/\Q_\cyc) =\underset{n}{\varinjlim}S_\dagger(A_{h\chi(-j)}[\pi^r]/\Q(\mu_{p^n})),
$$
where $S_\dagger(A_{h\chi(-j)}[\pi^r]/\Q(\mu_{p^n}))$ was defined in \S \ref{sec2}. We can explicitly write,
\begin{equation}\label{sel-general698979}
S_\mathrm{Gr}(A_{h\chi(-j)}/\Q_\cyc): =\mathrm{Ker}\Big(H^1(\Q_\Sigma/\Q_\cyc, A_{h\chi(-j)}) \lra \underset{q_\infty \in \Sigma_\infty}\prod \frac{H^1(\Q_{\cyc,q_\infty}, A_{h\chi(-j)})}{H_\mathrm{Gr}^1(\Q_{\cyc,q_\infty}, A_{h\chi(-j)})}\Big) 
\end{equation}
$$
\text{with }~H^1_\mathrm{Gr}(\Q_{\cyc,q_\infty}, A_{h\chi(-j)}) : =
\begin{cases} 
\text{Ker}\big(H^1(\Q_{\cyc,q_\infty}, A_{h\chi(-j)}) \lra H^1(I_{q_\infty}, A_{h\chi(-j)})\big) & \text{if } q_\infty \nmid p \\
\text{Ker}\big(H^1(\Q_{\cyc,q_\infty}, A_{h\chi(-j)}) \lra H^1(I_{q_\infty}, A^-_{h\chi(-j)})\big) & \text{if } q_\infty \mid  p.
\end{cases}
$$
Then   $S_\mathrm{Gr}( A_{{f_i}\chi(-j)}[\pi^r]/\Q_\cyc)$ is defined from \eqref{sel-general698979} using \eqref{sel-mod-r-general}. We have the following analogue of Theorem \ref{mainthmgr2}.
\begin{theorem}\label{theoremcyc1}
Let $p$ be an odd prime and for $i=1,2,$ let $f_i$ be a $p$-ordinary normalized cuspidal Hecke eigenform in $S_k(\Gamma_0(N), \epsilon_i)$, where $(N,p) =1$, $k \geq 2$. Let $r \in \N$ and $\phi : A_{f_1}[\pi^r] \lra A_{f_2}[\pi^r]$ be a $G_\Q$ linear isomorphism. We  assume the following:
\begin{enumerate}

\item $N$ is square-free and $\forall \ell \in S$,  $ \text{cond}_\ell(\bar{\rho}_{f_i})=\ell$ for $i=1,2$.
\item The condition $(\mathrm{C'_{i,\ell}})$, defined in equation \eqref{conduction-condition126}, is satisfied for $i=1,2$.

\item $p > k$.

\end{enumerate}
Then  for every quadratic character $\chi $ of $G_\Q$ and for every fixed $j$ with $0 \leq j \leq k-2$, we have an isomorphism 
of the $\pi^r ~\dagger-$Selmer groups for $\dagger \in \{\mathrm{BK}, \mathrm{Gr}\}$,
$$S_\dagger(A_{f_1\chi (-j)}[\pi^{r}]/\Q_\cyc) \cong S_\dagger( A_{f_2\chi (-j)}[\pi^{r}]/\Q_\cyc). $$
\end{theorem}
\proof  First of all, we note that it suffices to show for every $\chi$ and $0 \leq j \leq k-2$, 
\begin{equation}\label{cycgrasISD}
S_\mathrm{Gr}(A_{f_1\chi (-j)}[\pi^{r}]/\Q_\cyc) \cong S_\mathrm{Gr}( A_{f_2\chi (-j)}[\pi^{r}]/\Q_\cyc). 
\end{equation}
 Indeed,  as $t_1=t_2 =0 $ and $(N,p) =1$, the conditions (i), (ii) and (iii)' of \cite[Prop 4.2.30]{fk} is verified. Thus  for $i=1,2$, $$H^1_\mathrm{Gr}(\Q_{\cyc, q_\infty}, A_{f_i\chi (-j)}) \cong H^1_\mathrm{BK}(\Q_{\cyc, q_\infty}, A_{f_i\chi (-j)})$$ by  \cite[Prop 4.2.30]{fk} for $q_\infty \mid p$ as well as $q_\infty \mid q$ with $q \in \Sigma \setminus \{p\}$. Using the definition of $\pi^r$ Selmer group in \eqref{sel-mod-r-general}, we deduce for every $q_\infty \in \Sigma_\infty$, $$H^1_\mathrm{Gr}(\Q_{\cyc, q_\infty}, A_{f_i\chi (-j)}[\pi^r]) \cong H^1_\mathrm{BK}(\Q_{\cyc, q_\infty}, A_{f_i\chi (-j)}[\pi^r]).$$ Hence Theorem \ref{theoremcyc1} will follow once we establish \eqref{cycgrasISD}.  Note that in the proof of Theorem \ref{mainthmgr}, we have assumed $p> 2k-3$ to show that $A''^{I_p}_{{f_i}\chi(-j)}$ is $\pi$-divisible (see Corollary \ref{78914121881} and Remark \ref{div}). In this case of $\Q_\cyc$,  $A''^{I_{p_\infty}}_{{f_i}\chi(-j)}$ is $\pi$-divisible even without the assumption $p > 2k-3$, as we explain now. Note that $I_{p_\infty}$ acts on $A''_{{f_i}\chi(-j)}$  via $\bar{\omega}_p^{-j}\chi$. From the proof of Proposition \ref{ipandip'}, Case 1, we can see that $A''^{I_{p_\infty}}_{{f_i}\chi(-j)} =0$ unless  $j > 0$ and  $(\bar{\omega}_p^{-j}  \chi) _{\mid_{I_{p_\infty}}} =1~ (\text{mod } \pi)  $. Since $\bar{\omega}_p^{-j}  \chi$ has order prime to $p$, in the later case, we get $\bar{\omega}_p^{-j}  \chi =1 $ as a character of $I_{p_\infty}$. Consequently,  $A''^{I_{p_\infty}}_{{f_i}\chi(-j)} = A''_{{f_i}\chi(-j)}$ is $\pi$-divisible. 

Note the assumption $p >k$ is needed in the proof of Lemma \ref{phi''copy}.  Now the proof of \eqref{cycgrasISD} is very similar to the proof of Theorem \ref{mainthmgr} and hence omitted to avoid repetition. \qed

\medskip

 We denote the Teichm\"uller character $G_\Q \lra \mathbb F_p^\times \subset \Z_p^\times$ by  $\bar{\omega}_p$. Note that \eqref{cycgrasISD} is true if and only if $$S_\mathrm{Gr}(A_{f_1\chi{\bar{\omega}_p^{-j}}}[\pi^r] /\Q_\cyc)\cong S_\mathrm{Gr}( A_{f_2\chi {\bar{\omega}_p^{-j}}}[\pi^{r}]/\Q_\cyc).$$

\begin{rem}
\rm{If we assume $A_{f_i}[\pi] $ is an irreducible $G_\Q$  module, then as in Remark \ref{rem1first}, we can deduce 
\begin{equation}\label{913812}
S_\dagger(A_{f_i\chi {\bar{\omega}_p^{-j}}}/\Q_\cyc)[\pi^r] \cong S_\dagger( A_{f_i\chi{\bar {\omega}_p^{-j}}}[\pi^{r}]/\Q_\cyc),
\end{equation} for $\dagger \in \{\mathrm{Gr}, \mathrm{BK}\}$.
Thus using \eqref{913812} together with the hypotheses of Theorem \ref{theoremcyc1},  it follows that  for every quadratic  $\chi$ and every $0 \leq j \leq k-2$, we have an isomorphism 
\begin{equation}\label{dajjshdqudqi}
S_\dagger(A_{f_1\chi {\bar{\omega}_p^{-j}}}/\Q_\cyc)^\vee/{\pi^{r}} \cong S_\dagger( A_{f_2\chi {\bar{\omega}_p^{-j}}}/\Q_\cyc)^\vee/{\pi^{r}}, \quad \dagger \in \{\mathrm{BK}, \mathrm{Gr}\}. 
\end{equation}
 Here for a discrete module $\La: = O[[\Gamma]]\cong O[[T]]$ module $M$, we denote by $M^\vee $ the Pontryagin dual $\mathrm{Hom}_\mathrm{cont}(M, \Q_p/{\Z_p})$. By a deep theorem of Kato, as $f_i$ is $p$-ordinary we know $S_\mathrm{Gr}(A_{f_i\chi {\bar{\omega}_p^{-j}}}/\Q_\cyc)^\vee$ is a finitely generated torsion $\La$ module. Moreover, in this case (cf. \cite[Theorem 4.1.1]{epw}, \cite{gr}) $S_\mathrm{Gr}(A_{f_i\chi {\bar{\omega}_p^{-j}}}/\Q_\cyc)^\vee$ has no pseudonull (finite) $\La$-submodule. It then follows from \eqref{dajjshdqudqi} that for every quadratic character $\chi$ of $G_\Q$ and for  critical values $0 \leq j \leq k-2$,
  \begin{equation}\label{ajdsghqdlast}
 C_\La\Big(S_\mathrm{Gr}(A_{f_1\chi {\bar{\omega}_p^{-j}}}/\Q_\cyc)^\vee\Big) \equiv  C_\La\Big(S_\mathrm{Gr}(A_{f_2\chi {\bar{\omega}_p^{-j}}}/\Q_\cyc)^\vee\Big)  \quad (\text{mod } \pi^r)
 \end{equation}
 Here for a finitely generated torsion $\La$ module module $M$, we denote $C_\La(M)$, the characteristic ideal of $M$ in $\La$. 
  
    Under suitable condition, the $p$-adic $L$-function $ \mathcal L_{f_i\chi{\bar{\omega}^{-j}_p}}^p(T) \in \La$ exists (cf. \cite{epw}) and by  Iwasawa-Greenberg Main Conjecture $$ C_\La\Big(S_\mathrm{Gr}(A_{f_i\chi {\bar {\omega}_p^{-j}}}/\Q_\cyc)^\vee\Big) = (\mathcal L_{f_i\chi{\bar{\omega}^{-j}_p}}^p(T)),$$as ideals in $\La$. In his important work, Vatsal \cite{va} has shown that if $f_1 \equiv f_2 $ (mod $\pi^r$) then for any Dirichlet character  $\chi$ whose conductor is coprime to $N$, 
  \begin{equation}\label{vatsal-padiclfncon}
  \mathcal L_{f_1\chi{\bar{\omega}^{-j}_p}}^p(T) \equiv \mathcal L_{f_2\chi{\bar{\omega}^{-j}_p}}^p(T)  \quad \text{( mod  } \pi^r\La).
\end{equation}

In particular, using Iwasawa main conjecture
together with  \eqref{vatsal-padiclfncon}, one can obtain \eqref{ajdsghqdlast} for any quadratic $\chi$ whose conductor is coprime to $N$. Thus our Theorem \ref{theoremcyc1} can be thought of as an algebraic reflection of the congruence result of Vatsal via Iwasawa main conjecture. However, our Theorem \ref{theoremcyc1} is valid for all possible quadratic character $\chi$.
 }
 \end{rem}

\begin{rem}
\rm{Note that in the non-ordinary case i.e. when $a_p(f_i)$ is not a $p$-adic unit, $S_\mathrm{BK}(A_{f_i\chi (-j)}/\Q_\cyc)^\vee $ is not a  torsion $\La$ module i.e. it has positive $\La$ rank (cf. \cite{gr}).  For   weight $k~ ( >2) $ congruent cuspforms  which are good and non-ordinary at $p$, it is not clear how to  establish Theorem \ref{theoremcyc1} over $\Q_\cyc$ for the $\pi^r$ Bloch-Kato Selmer groups. 
}
\end{rem}

\section{Examples}\label{secexam}
In this section we give several numerical examples to illustrate all our main results.
\begin{example}\label{sjwjakkoi09k3kllm}
\begin{enumerate}
\item \textnormal{We consider the example of elliptic curves $1246B, 1246C$ considered in \cite[Table 1]{cm} and choose the prime  $p=5$. Let $f, g \in S_2(\Gamma_0(1246))$ be the primitive modular forms associated to $1246B$ and  $1246C$ respectively via modularity. We have $1246 = 2\times 7 \times 89$ is square-free and $5\nmid 1246$. Note the Fourier expansions of $f $ and $g$ are given by \cite{lmfdb}
$$ f(q) =  q-q^2+2q^3+q^4+2q^5-2q^6-q^7-q^8+q^9-2q^{10}+2q^{12} +O(q^{13}),$$
$$g(q) = q-q^2-3q^3+q^4-3q^5+3q^6-q^7-q^8+6q^9+3q^{10}-3q^{12}+ O(q^{13})$$
 By computing the minimal discriminant of $1246B$ and  $ 1246C$  and using \cite[Proposition 2.12(c)]{DDT} we can show  that $\forall \ell \in S=\{2, 7, 89\}$,  $ \text{cond}_\ell(\bar{\rho}_f)=\ell=\text{cond}_\ell(\bar{\rho}_g)$. Also using \cite{lmfdb} we get that $\bar{\rho}_f$ and $\bar{\rho}_g$ are irreducible and equivalent. Thus $f$ and $g$ satisfies all the hypotheses of Corollary \ref{cor} and Theorem \ref{mainthmgr2}. Hence $f$ and $g$ are  $5$-Bloch-Kato Selmer companion forms.}

\item \textnormal{Since $a_5(f)=2$ and $a_5(g)=-3$, we see that $5$ is a prime of  ordinary reduction for $f$ and $g$. Using Hida theory,  corresponding to $f$ and  $g$, there exists primitive forms $f_3$ and $g_3$ of weight $=3$, level $5N= 5\times 1246$ and nebentypus $\bar{\omega}_5^{-1}$, where $\bar{\omega}_5$ is the Teichm\"uller character, such that $f_3\equiv f\equiv g \equiv g_3$ (mod $\pi$) (cf. \cite{Wi}). Here $\pi$ is a prime ideal of $K_{f_3,g_3, \bar{\omega}_5}$ lying above $5$. As $k=3$, the condition [$p$-good] of Theorem \ref{mainthmgr2} does not apply although $5$ is a prime of good reduction of $f,g$. However, $a_5(h)  \neq \pm 1$ (mod $5$) for $h \in \{f, g\}$ and  hence also for $h \in \{f_3, g_3\}$. Thus the condition [$p$-general] of  Theorem \ref{mainthmgr2} hold.  We deduce the weight 3 forms $f_3$ and $g_3$ are $\pi$-Bloch-Kato Selmer companion.}

\item \textnormal{Again using Hida theory, by Remark \ref{finintefirst88237}, there are infinitely many cuspidal Hecke newforms $f_r\in S_2(\Gamma_0(1246\times 5^r), \psi_r)$  such that $(f,f_r)$ are $\pi$ Selmer companion.}

\item \textnormal{Using the extended definition of Selmer companion forms of different weights, from Corollary \ref{mainthmgr3}, we have that $f$ and $f_3$ are  $\pi$-Bloch-Kato Selmer companion and same is true for the pair $g$ and $g_3$.} 
\end{enumerate}
\end{example}

\begin{example}
\begin{enumerate}
\item 
\textnormal{ Consider the pair of modular forms $f = 127k4A$ and $g= 127k4C$ of level 127, weight $4$, trivial nebentypus and Galois orbits $A$ and $C$ respectively as appeared in \cite[Table 1]{dsw}. Here $K_f =\Q$ and $[K_g:\Q] = 17$. As the nebentypus is  trivial, in this case $K_{f,g,\epsilon}=K_g$. Then there exists a prime $\p$ of $K_g$ lying above the prime number  $p=43$ such that $f \equiv g $ mod $\p$ \cite[\S  7]{dsw}. The level $127$  is square free (a prime). Using \cite{sage} we have calculated, $a_{43}(f) = 80$ which is coprime to $43$. It follows that $f$ has good and ordinary reduction at $43$. Note that $43> 2k-3= 5$  
Since $f \equiv g $ mod $\p$ and $43 \nmid 127$, the same holds for $g$. Thus the conditions [$p$-ord] and $[p$-good] of Theorem \ref{mainthmgr2} holds. 
 Note that there are no newforms of weight $4$, level $1$ and trivial nebentypus. Then from the level lowering results of modular forms (by  Ribet, Serre 
{\it et.al}), we get that the prime to $p$  conductor of $\bar{\rho}_f=\bar{\rho}_g$ is not $1$. In particular, the condition $(1)$ of Theorem \ref{mainthmgr2} also holds. Also the nebentypus is trivial. Thus all the conditions  Theorem  \ref{mainthmgr2} are satisfied and we deduce $f$ and $g$ are $\mathfrak p$ Bloch-Kato-Selmer companion forms.}

\item \textnormal{Note that $a_{43}(f) \neq \pm 1$ (mod $43$). Using the Hida family passing through $f$ and $g$,  we can generate more examples of higher weight $\mathfrak p$ Selmer companion modular forms as in Example \ref{sjwjakkoi09k3kllm}(2), \ref{sjwjakkoi09k3kllm}(3).}
\end{enumerate}
\end{example}

\begin{example}\label{example159}
\begin{enumerate}
\item
\textnormal{
We take $f =159k4B$ and  $g= 159k4E \in S_4(\Gamma_0(153))$ with   trivial nebentypus.   Note $N=159 = 3 \times 53 $ is square-free.
The Fourier coefficients of $f$ belongs to $\Q$, on the other hand, $K_{f,g}=K_{g}$ with $[K_g:\Q] =16$. We take $p =5$ and using \cite{sage} compute that $a_{5}(f)=0.$ As $p=5 > k=4$, $5 \nmid 159$ and $a_{5}(f)=0$; the condition [$p$-ss] of Theorem \ref{mainthmgr2} is satisfied.
It was shown in  \cite{dsw} that there exists a prime $\p$ of $K_{g}$ lying above $p=5$ such that $f \equiv g$ (mod $\p$). Moreover,
in \cite[\S 7.2, paragraph 3]{dsw} it is given that there is no congruences between  $f$(respectively $g$) at $p$ with a newform of level dividing $N=159$. In particular, using level lowering result of modular forms (by  Ribet, Serre 
{\it et.al}) it follows that the hypothesis (1) on conductor of $\bar{\rho}_f=\bar{\rho}_g$ holds. Thus via Theorem \ref{mainthmgr2}, $f$ and $g$ are $\p$  Bloch-Kato Selmer companion forms.}

\item
\textnormal{
We again take the same forms $f =159k4B$ and  $g= 159k4E \in S_4(\Gamma_0(153))$.  However we now take $p =23 $ and using \cite{sage} compute that $a_{23}(f)=-49.$ 
It was shown in  \cite{dsw} that there exists a prime $\pi$ of $K_{g}$ lying above $p=23$ such that $f \equiv g$ (mod $\p$). As the nebentypus is trivial, and $a_{23}(f)=-49$, we can conclude that [$p$-good] and [$p$-ord] of Theorem \ref{mainthmgr2} are satisfied. As before, by  \cite[\S 7.2, paragraph 3]{dsw} there is no congruences between  $f$(resp $g$) at $p$ with a newform of level dividing $N=159$. In particular, using level lowering result of modular forms condition (1) holds. Thus via Theorem Theorem \ref{mainthmgr2}, $f$ and $g$ are again $\pi$ Bloch-Kato Selmer companion forms.}
\end{enumerate}
\end{example}

\begin{example}  \textnormal{ Next we consider the example of $f = 365k4A $ and $g = 365k4E \in S_4(\Gamma_0(365))$
with $N = 365 = 5 \times 73  $ and we choose $p=29$. Then $f$ has Fourier coefficients defined over $\Q$ and we compute via \cite{sage} $a_{29}(f) =-123$.
Also
$[K_{g}:\Q]=18$ and there exists a prime $\p$ of $K_{g}$ lying above $p=29$ such that $f \equiv g$ mod $\p$ \cite{dsw}.
It is given in \cite[\S 7.2, paragraph 3]{dsw}  there is no congruence between $f$ or $g$ at $\p$ with a newform of level dividing $365$. Thus using the same reasoning as in Example \ref{example159}, we conclude $f$ and $g$ are  $\p$  Bloch-Kato Selmer   companion forms.}
\end{example}

\begin{example} \textnormal{ Consider
$N = 453 = 3 \times 151$ and $f,g \in S_4(\Gamma_0(453))$ where $f =453k4A$ with Fourier coefficients in $\Q$ and $g=453k4E$ with $[K_g:\Q]=23$. For the prime $17$ we compute using \cite{sage} $a_{17}(f) =-66$. Once again from   \cite{dsw}, (i) $\exists$ a prime $\p$ of $K_{g}$ lying above $17$ such that $f \equiv g$ (mod $\p$) and (ii) there is no congruence   between $f$ (resp $g$) at $\p$ with a newform of level dividing $453$. By same reasoning as in Example \ref{example159}, $f$ and $g$ are $\p$ Bloch-Kato   Selmer companion.}
\end{example}

\thanks{{\bf Acknowledgement:} S. Jha acknowledges the support of ECR grant by SERB. D. Majumdar is supported by  IIT Madras NFIG grant
MAT/16-17/839/NFIG/DIPR.  S. Shekhar  acknowledges the support of DST INSPIRE Faculty grant. The authors would like to thank  David Loeffler and Antonio Lei for some discussions.}

\end{document}